\documentclass[11pt,a4paper]{amsart}
\usepackage[T1]{fontenc}
\usepackage[utf8]{inputenc}
\usepackage{lmodern}
\usepackage[margin=27mm]{geometry}
\usepackage{amsmath,amssymb,amsthm,mathtools}
\usepackage{booktabs}
\usepackage{microtype}
\usepackage{xcolor}
\usepackage{tikz}
\usetikzlibrary{arrows.meta}
\usepackage[unicode,colorlinks=true,linkcolor=blue!45!black,citecolor=blue!45!black,urlcolor=blue!45!black]{hyperref}
\hypersetup{pdftitle={Loop Space Models and Integral Homology of Connected Quandles},pdfauthor={Takefumi Nosaka},pdfsubject={Marked deloopings, localization, and integral homology}}
\numberwithin{equation}{section}
\theoremstyle{plain}
\newtheorem{theorem}{Theorem}[section]
\newtheorem{maintheorem}{Theorem}

\newtheorem{proposition}[theorem]{Proposition}
\newtheorem{lemma}[theorem]{Lemma}
\newtheorem{corollary}[theorem]{Corollary}
\theoremstyle{definition}
\newtheorem{definition}[theorem]{Definition}
\newtheorem{example}[theorem]{Example}
\theoremstyle{remark}
\newtheorem{remark}[theorem]{Remark}
\newcommand{\Z}{\mathbb Z}
\newcommand{\Q}{\mathbb Q}
\newcommand{\F}{\mathbb F}
\newcommand{\cM}{\mathcal{M}}
\DeclareMathOperator{\As}{As}
\DeclareMathOperator{\Inn}{Inn}
\DeclareMathOperator{\hofib}{hofib}
\DeclareMathOperator{\hocolim}{hocolim}
\DeclareMathOperator{\Cofib}{Cofib}
\DeclareMathOperator{\id}{id}
\title[Loop Space Models and Integral Homology]{Loop Space Models and Integral Homology of Connected Quandles}
\author{Takefumi Nosaka}
\address{Department of Mathematics, Institute of Science Tokyo, Japan}
\date{October 19, 2026}
\begin{document}

\begin{abstract}
We construct a marked integral delooping of a connected quandle's covering
monoid, realizing its rack space as a homotopy fiber.  For finite connected
quandles at primes $p$ not dividing a point-stabilizer order, we determine the
local $A_\infty$ type and the ordinary rack cohomology shuffle algebra over fields of
characteristic $p$.  We compute integral quandle homology in all positive
degrees for finite connected regular cyclic Alexander quandles and the
tetrahedral quandle.  At such primes dividing the quandle order, rack-space
homotopy groups have exponentially many cyclic direct summands of fixed
$p$-power order in each sufficiently large degree.
\end{abstract}
\maketitle

\section*{Introduction}

Quandles are algebraic structures that arise from conjugation in groups and colorings of knot diagrams.
The rack space $\mathcal{R}_X$ associated with a quandle $X$ has the rack homology complex as its cellular chain complex \cite{FRS}.
For an augmented rack, Clauwens equipped a covering of the associated rack space with a strictly associative monoid structure and showed that this monoid is equivalent to a loop space whenever it is connected \cite[Section~2.5]{Clauwens}.
Mostovoy established a comparison between this monoid and the geometric realization of an associated simplicial group \cite[Theorem~3.1]{Mostovoy}.
Building on these comparisons, we construct an explicit marked delooping from the augmentation data.
We apply its local form to integral quandle homology computations and to the study of the multiplication on the covering monoid.

For a connected quandle $X$ and $x_0\in X$, put $G=\Inn(X)$ and $H=\operatorname{Stab}_G(x_0)$, and write $\cM_X=B(G;X)$ for the covering monoid equipped with the Clauwens product.
We first construct $Y$ as a double mapping cylinder determined by $G,H$ and the right translation $S_{x_0}$.
The equivalence $B\cM_X\simeq Y$ is compatible with the maps from $BG$,
and $\mathcal{R}_X$ is recovered as the homotopy fiber of
$j:BG\to Y$ (Theorem~\ref{int:main}).
This integral model requires neither finiteness nor faithfulness.

In the finite case, for a prime $p$ satisfying $p\nmid|H|$, we obtain an equivalence of grouplike $A_\infty$-spaces
\[
 (\cM_X)_{(p)}
 \simeq_{A_\infty}
 \Omega\bigl((S^2\vee(BG)^\wedge_p)_{(p)}\bigr)
\]
(Theorem~\ref{loc:main} and Proposition~\ref{loc:completion}).
Here the subscript $(p)$ denotes $p$-localization, and the superscript ${}^\wedge_p$ denotes Bousfield--Kan $p$-completion.
We call a prime $p$ satisfying this condition a good prime.
At a good prime, the homotopy type of $(BG)^\wedge_p$ can be recovered from the $A_\infty$ type of $(\cM_X)_{(p)}$ (Theorem~\ref{ainf:classification}).

Combining this model with decompositions of the rack and quandle complexes, we compute quandle homology from the homology of an auxiliary loop space with coefficients in $\Z_{(p)}$.
Keeping track of both tensor and Tor terms allows us to determine the orders and multiplicities of the cyclic direct summands.
In the following two cases, the annihilation theorem
\cite[Corollary~6.2]{Nosaka2013} shows that, in degrees $n\ge2$,
only primes dividing $|X|$ need be considered when recovering
integral quandle homology from its localizations.
Throughout,~$H_n^Q(X;\Z)$ denotes integral quandle homology, and $C_m$ denotes the cyclic group of order $m$.
Our main computational results are as follows.

\begin{maintheorem}[Cyclic Alexander quandles]\label{intro:cyclic}
Let $p$ be an odd prime, let $r\ge1$ be an integer, and let $\omega\in(\Z/p^r\Z)^\times$ be a unit whose order $d>1$ divides $p-1$.
Equip $X=\Z/p^r\Z$ with the operation $x*y=\omega x+(1-\omega)y$.
Then
\[
 H_1^Q(X;\Z)\cong\Z,\qquad
 H_n^Q(X;\Z)\cong(C_{p^r})^{b_n}\quad(n\ge2)
\]
where the nonnegative integers $b_n$ are determined by
\[
 \sum_{n\ge2}b_nt^n
 =\frac{t^{2d-1}(1+t)}{1-t^{2d-2}-t^{2d-1}-t^{2d}}.
\]
\end{maintheorem}

In the case $r=1$, the author computed the integral quandle homology and gave cocycles representing generators of the cohomology with coefficients in $\F_p$ \cite[Theorem~1.1, Corollary~5.9]{Nosaka2013}.
Ishikawa also computed the quandle homology of connected dihedral quandles using a different spectral sequence \cite{Ishikawa}.
Theorem~\ref{intro:cyclic} extends the additive computation in \cite[Theorem~1.1]{Nosaka2013} to the above cyclic family of prime-power order.
The multiplicities $b_n$ are independent of $r$, whereas $p^r$ appears in the orders of the cyclic summands.
Corollary~\ref{cyc:regular-cyclic-integral} gives the corresponding
prime-by-prime formula for every finite connected cyclic Alexander quandle
$(A,T)$ satisfying $\gcd(|A|,\operatorname{ord}(T))=1$.
Such quandles are called regular in Section~\ref{sec:cyclic}.
In particular, the result applies to every dihedral quandle of odd order greater than one.
At a good prime $p$, the local computation also applies to any finite connected quandle whose inner automorphism group has a nontrivial cyclic Sylow $p$-subgroup (Theorem~\ref{cyc:sylow-homology}).

\begin{maintheorem}[The tetrahedral quandle]\label{intro:tetrahedral}
Let $\F_4=\F_2[\omega]/(\omega^2+\omega+1)$ and equip $X=\F_4$ with the operation $x*y=\omega x+(1-\omega)y$.
Then
\[
 H_1^Q(X;\Z)\cong\Z,\qquad
 H_n^Q(X;\Z)\cong C_2^{a_n}\oplus C_4^{b_n}\quad(n\ge2)
\]
where the nonnegative integers $a_n,b_n$ are determined by
\[
 \sum_{n\ge2}b_nt^n=\frac{t^3}{1-t-t^3},\qquad
 \sum_{n\ge2}(a_n+b_n)t^n=\frac{t^2(1+t)}{1-t-t^2-t^3}.
\]
\end{maintheorem}

Theorem~\ref{intro:tetrahedral} establishes the recurrence for the multiplicities of the $C_4$ direct summands conjectured by Niebrzydowski--Przytycki \cite[Example~13(ii)]{NiebrzydowskiPrzytycki2010}.
Their prediction for the entire group, however, requires a correction: in degree $8$, the group is $H_8^Q(X;\Z)\cong C_2^{33}\oplus C_4^4$.

For a finite connected quandle $X$, a good prime $p$, and a field $k$ of characteristic $p$, the marked model also determines the ordinary rack cup product: Proposition~\ref{hom:rack-shuffle} gives a noncanonical graded algebra isomorphism
\[
 H_R^*(X;k)\cong\operatorname{Sh}_k\bigl(sH^*(\Omega L;k)\bigr),
 \qquad L=\Cofib(BH\to BG),
\]
where $\operatorname{Sh}_k(V)$ denotes the graded shuffle algebra on $V$, $s$ raises degree by one, and the shuffle signs use the suspended degrees.
The unit of $H^*(\Omega L;k)$ supplies a letter in degree one.
This is a field-coefficient statement about the ordinary rack cup product; the integral additive computations retain the Tor terms separately.

For a finite connected quandle $X$ and a good prime $p$, the specified multiplication on the $p$-localization $(\widetilde{\cM_X})_{(p)}$ of the universal cover admits a refinement to an $E_2$ structure if and only if $p\nmid|X|$ (Theorem~\ref{ainf:universal-cover-e2}).
When $p\mid|X|$, iterated Samelson products yield torsion direct summands in the homotopy groups of $(\cM_X)_{(p)}$ in every sufficiently large degree.
The Hall-word decomposition strengthens this to an exponential lower bound for the number of cyclic direct summands of some fixed $p$-power order in each sufficiently large degree, and hence also for their cumulative number (Proposition~\ref{ainf:iterated-samelson} and Theorem~\ref{ainf:general-exponential-torsion}).

Section~\ref{sec:preliminaries} sets out the definitions and conventions.
Section~\ref{int:section} constructs the integral model, and Section~\ref{loc:section} treats localization, completion, and the classification of $A_\infty$ types.
After discussing fundamental groups, universal covers, and rationalization in Section~\ref{rat:section}, we develop the homology formulas in Section~\ref{hom:section} and prove Theorems~\ref{intro:cyclic} and~\ref{intro:tetrahedral} in Sections~\ref{sec:cyclic} and~\ref{tet:section}.
Section~\ref{ex:comparisons} illustrates the tensor and Tor terms with a product Alexander quandle of order $49$.
Section~\ref{ex:nonalex-section} gives local computations for
transposition quandles and computes the rack and positive-degree quandle
homology of rank-one symplectic quandles over finite fields of order greater
than $27$, away from the defining characteristic.
Section~\ref{ainf:section} is devoted to applications to multiplicative structures and unstable homotopy theory.
Appendix~\ref{app:hall} supplies the Hall-word count and the asymptotic estimate used in the degreewise growth bound.

\section{Augmented quandles and their rack spaces}\label{sec:preliminaries}

We work in the category of compactly generated spaces and use CW models when necessary.
Unless a continuous strict monoid homomorphism is specified, an equivalence of monoids means an equivalence of grouplike $A_\infty$-spaces.
For a localization $R$ of $\Z$ and a nilpotent space $Z$, let $Z_R$ denote the corresponding localization; we write $Z_{\Z_{(p)}}=Z_{(p)}$ and $Z_{\Q}$ for $p$-localization and rationalization, respectively.
The notation $Z_p^\wedge$ denotes Bousfield--Kan $p$-completion and is also used for classifying spaces of discrete groups.
All quandles and augmentation groups are discrete, and $C_n$ denotes the cyclic group of order $n$.

\subsection{Right actions and augmentation data}

A \emph{quandle} is a nonempty set $X$ with an operation $x*y$ such that $x*x=x$, every right translation $S_y:x\mapsto x*y$ is bijective, and $(x*y)*z=(x*z)*(y*z)$ holds.
We use right actions $x^{gk}=(x^g)^k$ throughout.
Accordingly, products of permutations acting on the right are applied from left to right.
The \emph{inner automorphism group} $\Inn(X)$ is generated by the permutations $S_y$.
A quandle is \emph{connected} if this group acts transitively, and \emph{faithful} if $y\mapsto S_y$ is injective.
Its associated group (or adjoint group) is
\begin{equation}\label{pre:associated-group}
 \As(X)=\langle e_x\ (x\in X)\mid e_y^{-1}e_xe_y=e_{x*y}\rangle.
\end{equation}
The group homomorphism $\As(X)\to\Inn(X)$ sends $e_x$ to $S_x$.

An \emph{augmented quandle} consists of a group $G$, a right $G$-set $X$, and a set map $\eta:X\to G$ satisfying
\begin{equation}\label{pre:augmentation}
 \eta(x^g)=g^{-1}\eta(x)g,
 \qquad x*y=x^{\eta(y)},
 \qquad x^{\eta(x)}=x.
\end{equation}
We mainly use the canonical augmentation $G=\Inn(X)$, $\eta(x)=S_x$ of a connected quandle.
In Sections~\ref{sec:preliminaries} and~\ref{int:section}, the first part of Section~\ref{loc:section}, and Subsections~\ref{low:torus-subsection} and~\ref{low:covers-subsection}, we work more generally with transitive augmentations for which $\eta(X)$ generates $G$.
We retain the chosen action and augmentation map and assume neither finiteness nor finite generation.
Finiteness hypotheses are stated explicitly in computations and comparisons with completion that require them.

Fix $x_0\in X$, put $H=\operatorname{Stab}_G(x_0)$ and $h=\eta(x_0)$, and write $i:H\hookrightarrow G$ for the inclusion homomorphism.
By \eqref{pre:augmentation}, we have $h\in Z(H)$, and $Hg\mapsto x_0^g$ gives an identification $H\backslash G\cong X$ of right $G$-sets.
Under this identification,
\begin{equation}\label{pre:coset-operation}
 Hg*Hk=Hgk^{-1}hk,\qquad \eta(Hg)=g^{-1}hg.
\end{equation}
Since $h\in Z(H)$, these formulas are independent of the coset representatives.
The generation hypothesis implies that the normal closures of both $h$ and $H$ are $G$.

\begin{remark}\label{pre:stabilizer-not-centralizer}
The stabilizer satisfies $H\subseteq C_G(h)$, but equality is not assumed.
Nor do we assume that $\eta$ is injective; thus nonfaithful quandles are included.
\end{remark}

\subsection{Cubical spaces and the monoid product}

For a right $G$-set $S$, define a cubical set without degeneracies by $B(S;X)_n=S\times X^n$.
The two face maps in the $i$th direction are the set maps
\begin{align}
 d_i^0(s;x_1,\ldots,x_n)
   &=(s;x_1,\ldots,\widehat{x_i},\ldots,x_n),\label{pre:faces0}\\
 d_i^1(s;x_1,\ldots,x_n)
   &=(s\eta(x_i);x_1*x_i,\ldots,x_{i-1}*x_i,
                 x_{i+1},\ldots,x_n).\label{pre:faces1}
\end{align}
We also write $B(S;X)$ for its cubical realization.
The rack space is $\mathcal{R}_X=B(*;X)$; it has a single vertex.
For a commutative coefficient ring $\Lambda$, extending the face maps $\Lambda$-linearly gives the boundary homomorphism $\partial_n$; the cellular chain complex of $\mathcal{R}_X$ is the rack complex
\begin{equation}\label{pre:rack-chain}
 C_n^R(X;\Lambda)=\Lambda[X^n],\qquad
 \partial_n=\sum_{i=1}^n(-1)^i(d_i^0-d_i^1),\qquad C_0^R(X;\Lambda)=\Lambda.
\end{equation}
We define $C_*(S;X;\Lambda)$ using the same conventions.
Thus $H_n^R(X;\Lambda)=H_n(\mathcal{R}_X;\Lambda)$, including degree zero.
Since the faces of $\mathcal{R}_X$ in direction $1$ agree, we have $\partial_1=0$.

For the specified augmentation, put $\cM=B(G;X)$; for the canonical augmentation, we also write $\cM_X=B(\Inn(X);X)$.
Put $\mathcal{N}_X=B(X;X)$.
The notation $x^{g}$ denotes the action on every component of a tuple.
The formula for the continuous multiplication map $\cM\times \cM\to \cM$,
\begin{multline}\label{pre:monoid-product}
 [g;t_1,x_1,\ldots,t_m,x_m]\,
 [k;s_1,y_1,\ldots,s_n,y_n]\\
 =[gk;t_1,x_1^k,\ldots,t_m,x_m^k,s_1,y_1,\ldots,s_n,y_n]
\end{multline}
defines a strictly associative product on $\cM$ with unit the vertex $1\in G$.
The identities $\eta(x^k)=k^{-1}\eta(x)k$ and $(x*y)^k=x^k*y^k$ show that this formula is compatible with the cubical face maps.
Associativity follows from $x^{gk}=(x^g)^k$.
This is the Clauwens product \cite[Sections 2.4--2.6]{Clauwens}.

The realization has CW type even for an arbitrary set of cubes, and the unit vertex is a closed CW subcomplex; hence $\cM$ is well-pointed.

The vertices form a subgroup $G\subset \cM$.
Left multiplication by a vertex $a\in G$ sends $[g;\mathbf t,\mathbf x]$ to $[ag;\mathbf t,\mathbf x]$.
This is the action by deck transformations of the regular covering $\cM\to \mathcal{R}_X$.
More generally, $B(S;X)\cong S\times_G \cM$, where the quotient pairs the right action on $S$ with this left action on $\cM$.
In particular, the covering map $\cM\to\mathcal{N}_X$ and the induced homeomorphism of quotient spaces are given by
\begin{equation}\label{pre:intermediate-cover}
 \mathcal{N}_X\cong H\backslash \cM,
 \qquad [g;\mathbf t,\mathbf x]\longmapsto
          [x_0^g;\mathbf t,\mathbf x].
\end{equation}
The covering property follows by lifting a cube with a prescribed initial vertex \cite[Section 2.5]{Clauwens}\cite[Section 3]{FRS}.

The edge $[g;t,x]$ joins $g$ to $g\eta(x)$.
Thus $\cM$ is path-connected and hence grouplike, and the space $\cM$ is simple: $\pi_1(\cM)$ is abelian and acts trivially on the higher homotopy groups.
Write $BZ=|[n]\mapsto Z^n|$ for the classifying space of the topological monoid $Z$.
For a grouplike $A_\infty$-space $Z$, $BZ$ denotes a
delooping of its specified $A_\infty$ structure,
understood up to based homotopy equivalence.
\phantomsection\label{pre:bar-loop}
The bar--loop comparison for connected, well-pointed monoids of CW type gives an $A_\infty$ equivalence $\cM\simeq_{A_\infty}\Omega B\cM$ with respect to the specified product \cite[Propositions~4.17(1), 5.4 and Theorem~5.6]{Vogt2015}.
Moreover, the two-cell $(x,y)$ of $\mathcal{R}_X$ gives the relation $e_xe_y=e_ye_{x*y}$.
Using the group homomorphism $\As(X)\to G$ defined by $e_x\mapsto\eta(x)$, we therefore obtain from \eqref{pre:associated-group} and covering space theory that
\begin{equation}\label{pre:fundamental-cover}
 \pi_1(\mathcal{R}_X)=\As(X),\qquad
 \pi_1(\cM)=\ker(\As(X)\to G).
\end{equation}

For every $n\ge2$, the action of $\pi_1(\mathcal{R}_X)=\As(X)$ on $\pi_n(\mathcal{R}_X)$ is trivial \cite[Proposition~5.2]{FRSRack}.
This also follows from the Clauwens monoid structure on the universal cover $B(\As(X);X)$ obtained by taking augmentation group $\As(X)$ and augmentation $x\mapsto e_x$ \cite[Proposition~2.16]{Clauwens}.
Indeed, deck transformations are left multiplications by vertices, and left multiplication along a path from the unit to a vertex gives a homotopy to the identity compatible with change of basepoint.

\begin{definition}\label{pre:good-prime}
When $H$ is finite and a prime $p$ satisfies $p\nmid |H|$, we say that $p$ is \emph{good for the chosen transitive augmentation}.
\end{definition}

\section{An integral model with its classifying map}
\label{int:section}

We use the generating transitive augmentation and the notation $G,H,h,i,\cM$ of Section~\ref{sec:preliminaries}, without imposing finiteness.

Define group homomorphisms
\begin{equation}\label{int:two-maps}
 a:H\times\Z\longrightarrow G,\quad a(u,m)=i(u),
 \qquad
 b_h:H\times\Z\longrightarrow H,\quad b_h(u,m)=uh^m.
\end{equation}
The condition $h\in Z(H)$ ensures that $b_h$ is a group homomorphism.
Using the induced maps $Ba,Bb_h$ on classifying spaces, define the homotopy pushout $Y$ and the inclusion $j$ by
\begin{equation}\label{int:Y-definition}
 Y=\hocolim\bigl[BG\xleftarrow{Ba}B(H\times\Z)
                         \xrightarrow{Bb_h}BH\bigr],
 \qquad j:BG\longrightarrow Y.
\end{equation}
We take the double mapping cylinder as the model for $Y$, marked by the inclusion $j$ at the left endpoint; see Figure~\ref{fig:double-mapping-cylinder}.

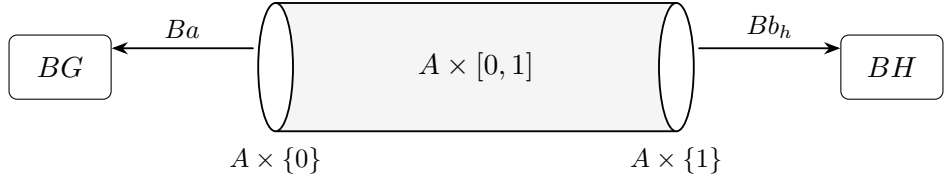
\begin{figure}[htbp]
\centering
\begin{tikzpicture}[>=Stealth,line cap=round,line join=round,
  every node/.style={font=\small}]
 \fill[black!4] (-2.65,-.85) rectangle (2.65,.85);
 \draw[line width=.7pt] (-2.65,.85)--(2.65,.85)
   (-2.65,-.85)--(2.65,-.85);
 \draw[line width=.7pt,fill=white] (-2.65,0) ellipse (.23 and .85);
 \draw[line width=.7pt,fill=white] (2.65,0) ellipse (.23 and .85);
 \node[font=\normalsize] at (0,0) {$A\times[0,1]$};
 \node[below=2pt] at (-2.65,-.85) {$A\times\{0\}$};
 \node[below=2pt] at (2.65,-.85) {$A\times\{1\}$};
 \node[draw,rounded corners=3pt,minimum width=1.35cm,
   minimum height=.85cm,fill=white,font=\normalsize] (g) at (-5.5,0) {$BG$};
 \node[draw,rounded corners=3pt,minimum width=1.35cm,
   minimum height=.85cm,fill=white,font=\normalsize] (h) at (5.5,0) {$BH$};
 \draw[->,line width=.7pt] (-2.95,.25)--node[above] {$Ba$}
   (g.east |- 0,.25);
 \draw[->,line width=.7pt] (2.95,.25)--node[above] {$Bb_h$}
   (h.west |- 0,.25);
\end{tikzpicture}
\caption{Schematic double mapping cylinder $Y$, with
$A=B(H\times\Z)$. The arrows indicate the attachment maps; the
marking $j:BG\hookrightarrow Y$ is the inclusion of the $BG$
component.}
\label{fig:double-mapping-cylinder}
\end{figure}

\begin{theorem}[Integral model]\label{int:main}
Let $X$ be any connected discrete quandle with its canonical augmentation, or, more generally, with a transitive augmentation as specified above.
The double mapping cylinder $Y$ in \eqref{int:Y-definition} is a simply connected CW complex, and there is a based homotopy equivalence under $BG$
\[
 B \cM\simeq Y.
\]
Here the continuous map $BG\to B\cM$ is induced by the inclusion of the vertex group $G\subset \cM=B(G;X)$.
Consequently, there is an equivalence of grouplike $A_\infty$-spaces
\begin{equation}\label{int:loop-model}
 \cM\simeq_{A_\infty}\Omega Y.
\end{equation}
For any classifying map $\gamma:\mathcal{R}_X\to BG$ of the covering $\cM\to \mathcal{R}_X$, there is an equivalence in the homotopy category of spaces over $BG$
\begin{equation}\label{int:rack-fibre}
 (\mathcal{R}_X,\gamma)\simeq
 \bigl(\hofib(j:BG\to Y),\ \hofib(j)\longrightarrow BG\bigr).
\end{equation}
If $G$ is finite, then $Y$ admits a CW model of finite type.
\end{theorem}

An equivalence over $BG$ means a zigzag of maps over $BG$ whose underlying maps are homotopy equivalences.

The proof may be skipped on a first reading: subsequent sections use the model $(Y,j)$ and the conclusions of the theorem, but not the auxiliary constructions in its proof.

\begin{proof}[Proof of Theorem~\ref{int:main}]
We first describe the simplicial group used in the comparison.
For $n\ge0$, let $\mathcal E_n$ be the discrete group generated by $G$ and $u_j(x)$, $x\in X$, $0\le j<n$, subject only to the relations
\begin{equation}\label{int:E-presentation}
 \text{the relations of $G$},\qquad
 g^{-1}u_j(x)g=u_j(x^g).
\end{equation}
In particular, $\mathcal E_0=G$.
The face maps $d_k:\mathcal E_n\to\mathcal E_{n-1}$ ($n\ge1$) and degeneracy maps $s_k:\mathcal E_n\to\mathcal E_{n+1}$ ($n\ge0$), $0\le k\le n$, fix $G$ and satisfy
\[
 d_k u_j(x)=
 \begin{cases}
  1,&k=j=0,\\
  u_{j-1}(x),&k\le j,\ j>0,\\
  u_j(x),&k>j,\ j<n-1,\\
  \eta(x),&k=n,\ j=n-1,
 \end{cases}
 \qquad
 s_k u_j(x)=
 \begin{cases}
  u_{j+1}(x),&k\le j,\\
  u_j(x),&k>j.
 \end{cases}
\]
Equivariance of $\eta$ ensures that these formulas preserve the relations.
To see the simplicial identities, index $u_j(x)$ by the nonconstant simplex of the standard simplicial interval $\Delta[1]$ that is $0$ on $0,\ldots,j$ and $1$ on $j+1,\ldots,n$; the constant simplices at $1$ and $0$ represent $1$ and $\eta(x)$, respectively.
The formulas are then induced by the simplicial operators of $\Delta[1]$.
They define a simplicial group $\mathcal E_\bullet$.
Write
\[
 \mathcal E=|\mathcal E_\bullet|
\]
for its geometric realization.
The degreewise multiplication induces a topological group structure, and the constant simplicial subgroup $G$ gives the inclusion $G\hookrightarrow\mathcal E$.
The simplicial group $\mathcal E_\bullet$ is the one constructed in \cite[Section~2.1]{Mostovoy} for the augmentation extended to the free group on $X$.
The canonical group-completion comparison of \cite[Theorem~3.1 and its proof]{Mostovoy} gives a continuous strict monoid homomorphism $\cM\to\mathcal E$ that is a homotopy equivalence and fixes $G$.
Our edge labels $(g,x)$ correspond to $(x^{g^{-1}},g)$ in that construction, preserving the vertices and their left and right actions.
For $x\in X$ and $t\in[0,1]$, put $\varepsilon_x(t)=[1;t,x]$.
The cubical product gives
\[
 \varepsilon_x(0)=1,\qquad \varepsilon_x(1)=\eta(x),\qquad
 \varepsilon_x(t)k=k\varepsilon_{x^k}(t)\quad(k\in G),
\]
and
\[
 [g;t_1,x_1,\ldots,t_n,x_n]
 =g\varepsilon_{x_1}(t_1)\cdots\varepsilon_{x_n}(t_n).
\]
These identities recover both face maps in \eqref{pre:faces0}--\eqref{pre:faces1}.
Under the indicated identification of edges and vertex actions, the cubical product agrees with the product in Mostovoy's model.

Let $\mathsf F_n$ be the free group on $v_0,\ldots,v_{n-1}$, with $\mathsf F_0=1$.
Transitivity gives an isomorphism
\begin{equation}\label{int:amalgam-map}
 G*_H(H\times\mathsf F_n)\xrightarrow{\cong}\mathcal E_n,
 \qquad g\longmapsto g,\quad v_j\longmapsto u_j(x_0).
\end{equation}
Indeed, $H$ commutes with $u_j(x_0)$, and the inverse fixes $G$ and sends $u_j(x_0^g)$ to $g^{-1}v_jg$.
This is independent of the representative because $v_j$ commutes with $H$.

Use nerve models for classifying spaces, and put $A=B(H\times\Z)$.
The section $s:H\to H\times\Z$, $s(u)=(u,0)$, identifies $BH$ with a CW subcomplex of $A$.
Define
\begin{equation}\label{int:R-definition}
 R_n=BG\cup_{BH}A_0\cup_{BH}\cdots\cup_{BH}A_{n-1},
 \qquad R_0=BG,
\end{equation}
where each $A_j$ is a copy of $A$, attached along $BH$ by $Bs$ and $Bi$.
Index $A_j$ by the same nonconstant simplex of $\Delta[1]_n$ as $u_j(x_0)$.
The simplicial operators are induced by those of $\Delta[1]$, with the constant simplices at $1$ and $0$ mapped to $BG$ by $Ba$ and $Bi\circ Bb_h$, respectively; all operators fix $BG$.
These maps agree on $BH$ since $a\circ s=i$ and $b_h\circ s=\id_H$.

The homomorphisms $(u,m)\mapsto i(u)u_j(x_0)^m$ from $H\times\Z$ to $\mathcal E_n$, together with $G\hookrightarrow\mathcal E_n$, induce maps $R_n\to B\mathcal E_n$ fixing $BG$.
They are homotopy equivalences by the classifying-space theorem for amalgamated products \cite[Theorem~1B.11 and Example~1B.12]{Hatcher}, applied to the injections $H\hookrightarrow G$ and $H\hookrightarrow H\times\Z$, together with \eqref{int:amalgam-map}.
The CW inclusions in \eqref{int:R-definition} make the ordinary pushout a homotopy pushout.
The maps form a simplicial map by the displayed formulas, including $d_0u_0(x_0)=1$ and $d_nu_{n-1}(x_0)=h$.
The degeneracy maps of $R_\bullet$ fix $BG$ and embed each source copy of $A$ as a target copy, while those of $B\mathcal E_\bullet$ are induced by split injections of groups.
Thus, in both simplicial spaces, the degeneracy images and their unions are CW subcomplexes.
Segal's realization theorem \cite[Appendix~A, Proposition~A.1(ii),(iv)]{Segal1974} and compatibility with finite products \cite[Theorem~11.5]{May} therefore give
\[
 |R_\bullet|\simeq|B\mathcal E_\bullet|\cong B\mathcal E
\]
under $BG$.

Put $I=[0,1]$.
The realization $|R_\bullet|$ is obtained from $Y$ by collapsing each interval of the section cylinder $BH\times I$ determined by $Bs$ to its endpoint in $BG$; here the $BG$ endpoint corresponds to vertex $1$ of $\Delta[1]$.
This cylinder is a CW subcomplex, and its projection onto $Bi(BH)$ is a homotopy equivalence.
The gluing lemma gives a homotopy equivalence $Y\to|R_\bullet|$ fixing $BG$.
Since $\cM$ and $\mathcal E$ are well-pointed monoids of CW type, the comparison $\cM\to\mathcal E$ also induces a homotopy equivalence $B\cM\to B\mathcal E$.
We therefore obtain
\[
 Y\xrightarrow{\simeq}B\mathcal E
   \xleftarrow{\simeq}B\cM.
\]
Since the inclusions of $BG$ into $Y$ and $B\mathcal E$ are CW cofibrations, the first map admits a homotopy inverse under $BG$.
Thus $B\cM\simeq Y$ under $BG$.
The \hyperref[pre:bar-loop]{bar--loop comparison of Section~\ref{sec:preliminaries}} gives \eqref{int:loop-model}; since $\cM$ is connected, $\pi_1Y\cong\pi_0\cM=0$.
The nerve maps are cellular, so the defining double mapping cylinder is a CW complex.

Let $p_{EG}:EG\to BG$ be a universal principal right $G$-bundle.
For a left $G$-space $Z$, write $EG\times_G Z$ for the quotient of $EG\times Z$ by $(\xi g,z)\sim(\xi,gz)$, with projection $[\xi,z]\mapsto p_{EG}(\xi)$ to $BG$.
For the specified classifying map $\gamma:\mathcal{R}_X\to BG$, choose a classifying bundle map $e:\cM\to EG$ satisfying
\[
 e(gm)=e(m)g^{-1},\qquad p_{EG}e(m)=\gamma([m]).
\]
Then $[m]\mapsto[e(m),m]$ defines a map $\mathcal{R}_X\to EG\times_G\cM$ over $BG$.
Its left inverse is the bundle projection to $\mathcal{R}_X$, with contractible fiber $EG$; since $\mathcal{R}_X$ is a CW complex, this projection and hence the map from $\mathcal{R}_X$ are homotopy equivalences.
The $G$-equivariant comparison $\cM\to\mathcal E$ induces a map of associated bundles over $BG$ that is a homotopy equivalence on every fiber, hence on total spaces.
We obtain equivalences
\[
 \mathcal{R}_X\simeq EG\times_G\cM
   \simeq EG\times_G\mathcal E
   \simeq\hofib(BG\to B\mathcal E)
   \simeq\hofib(j)
\]
over $BG$.
Here $EG\times_G\mathcal E$ is the principal $\mathcal E$-bundle induced by $G\hookrightarrow\mathcal E$, and the last equivalence uses the marked comparison with $Y$.
This proves \eqref{int:rack-fibre}.

Finally, if $G$ is finite, then $H$ is finite.
Finite-type CW models for $BG$, $BH$, and $B(H\times\Z)\simeq BH\times S^1$, together with cellular representatives of the two maps, give a finite-type CW model for $Y$.
\end{proof}

\section{Localization, completion, and classification of local monoids}
\label{loc:section}

We use the transitive augmentation and marked delooping $Y$ of Theorem~\ref{int:main}.
For the moment, no finiteness assumption is imposed.
Let $L$ be the cofiber of the cellular map $Bi$ induced by the inclusion homomorphism $i$, and let $\ell$ be the cellular quotient map:
\begin{equation}\label{loc:L-definition}
 L=\Cofib(BH\xrightarrow{Bi}BG),\qquad
 \ell:BG\longrightarrow L.
\end{equation}
We construct these maps using the nerves of groups and take the quotient $L=BG/BH$ by the CW subcomplex $BH\subset BG$.

This homotopy cofiber was used for Alexander quandles in \cite[Section~4.2, Definition~4.6, Proposition~4.7]{Clauwens}, where $H$ is the linear cyclic subgroup.
Here we use the full point stabilizer and compare the cofiber with the marked delooping.

\begin{lemma}\label{loc:L-properties}
The space $L$ is simply connected.
If $R$ is a localization of $\Z$ and $\widetilde H_*(BH;R)=0$, then $\ell$ induces an isomorphism in homology with coefficients in $R$.
In particular, this condition holds if $H$ is finite and $|H|$ is invertible in $R$.

If $G$ is finite, then $L$ admits a CW model with finitely many cells in each dimension, $\widetilde H_*(L;\Q)=0$, and $\pi_n(L)$ is finite for every $n\ge2$.
\end{lemma}

\begin{proof}
Since the conjugates of $h\in H$ generate $G$, van Kampen's theorem gives $\pi_1(L)=G/\langle\!\langle H\rangle\!\rangle=1$.
The homological assertion follows from the cofiber exact sequence and transfer.
If $G$ is finite, the quotient of nerves $BG/BH$ has finitely many cells in each dimension, and Serre's finite generation theorem \cite[Theorem~4.5.4]{MayPonto} implies that each $\pi_n(L)$ is finitely generated.
Moreover, $L$ is simply connected and rationally acyclic, so $L_{\Q}\simeq *$; these groups are therefore torsion and hence finite.
\end{proof}

\begin{theorem}\label{loc:main}\label{loc:comparison}
There is a based map $\Phi:Y\to L\vee S^2$ such that the composite $\Phi j$ agrees with $BG\xrightarrow{\ell}L\hookrightarrow L\vee S^2$ up to based homotopy.
For any localization $R\subset\Q$ of $\Z$, if $\widetilde H_*(BH;R)=0$, then $\Phi$ induces an isomorphism in homology with coefficients in $R$ and a based homotopy equivalence
\begin{equation}\label{loc:Y-equivalence}
 Y_R\simeq(S^2\vee L)_R.
\end{equation}
This is an equivalence under $BG$ with respect to $BG\xrightarrow{j}Y\to Y_R$ and $BG\xrightarrow{\ell}L\hookrightarrow S^2\vee L\to(S^2\vee L)_R$.
If $\cM_R$ is equipped with the localized $A_\infty$ structure represented by $\Omega(Y_R)$, there is also an $A_\infty$ equivalence
\begin{equation}\label{loc:M-equivalence}
 \cM_R\simeq_{A_\infty}\Omega\bigl((S^2\vee L)_R\bigr).
\end{equation}
In particular, when $H$ is finite, this applies to $R=\Z_{(p)}$ for $p\nmid|H|$ and to $R=\Z[1/|H|]$.
Neither $G$ nor $X$ need be finite.
\end{theorem}

\begin{proof}
Let $r:H\times\Z\to\Z$ be the second projection.
The following diagram of spans commutes strictly:
\[
\begin{array}{ccccc}
 BG&\xleftarrow{Ba}&B(H\times\Z)&\xrightarrow{Bb_h}&BH\\
 {\scriptstyle \ell}\downarrow&&{\scriptstyle Br}\downarrow&&\downarrow\\
 L&\xleftarrow{\ *\ }&B\Z&\xrightarrow{\ *\ }&*
\end{array}
\]
Indeed, $Ba$ factors through $BH$, whose image is collapsed by $\ell$.
Taking homotopy pushouts gives $\Phi:Y\to L\vee\Sigma B\Z\simeq L\vee S^2$ with the specified restriction.
By Lemma~\ref{loc:L-properties} and the K\"unneth theorem, all three vertical maps induce isomorphisms in homology with coefficients in $R$.
Here we use $B(H\times\Z)\simeq BH\times B\Z$ and the freeness of $H_*(B\Z;R)$.
Comparison of the Mayer--Vietoris exact sequences therefore shows that $\Phi$ induces an isomorphism in homology with coefficients in $R$.

Since $Y$ and $S^2\vee L$ are simply connected, the localization theorem \cite[Theorem~6.1.2]{MayPonto} gives \eqref{loc:Y-equivalence}.
The restriction of $\Phi$ and naturality of localization show that this is an equivalence under $BG$.
The required homotopy-commutative diagrams can be made strictly commutative using mapping cylinders.
Theorem~\ref{int:main} and localization of loop spaces \cite[Proposition~6.2.4]{MayPonto} give $\cM_R\simeq_{A_\infty}\Omega(Y_R)$, so looping \eqref{loc:Y-equivalence} yields \eqref{loc:M-equivalence}.
\end{proof}

\subsection{Identifying the localized homotopy cofiber}
\label{loc:completion-subsection}

In this subsection, $G$ is a finite group; hence $H$ and $X$ are also finite.
Since $BG$ need not be nilpotent, we use its Bousfield--Kan $p$-completion $BG_p^\wedge$ rather than its localization.

\begin{proposition}\label{loc:completion}\label{loc:completed-group-model}
If $G$ is finite and $p\nmid |H|$, there is a homotopy equivalence
\begin{equation}\label{loc:L-completion}
 L_{(p)}\simeq BG_p^\wedge.
\end{equation}
It may be chosen as an equivalence in the homotopy category under $BG$ with respect to $BG\xrightarrow{\ell}L\to L_{(p)}$ and the completion map $BG\to BG_p^\wedge$.
Moreover, there is an $A_\infty$ equivalence
\begin{equation}\label{loc:completed-group-formula}
 \cM_{(p)}\simeq_{A_\infty}\Omega\bigl((S^2\vee BG_p^\wedge)_{(p)}\bigr)
\end{equation}
and if $p\nmid |G|$, then $\cM_{(p)}\simeq_{A_\infty}\Omega(S^2_{(p)})$.
\end{proposition}

\begin{proof}
Lemma~\ref{loc:L-properties} and the completion exact sequence \cite[Theorem~2.6]{BarthelBousfield} give $\pi_n(L_p^\wedge)\cong\pi_n(L)\otimes\Z_{(p)}$ for $n\ge2$, and $\pi_1(L_p^\wedge)=0$.
Thus $L_p^\wedge$ is $p$-local, and the completion map factors as $L\to L_{(p)}\xrightarrow{\kappa}L_p^\wedge$.
The map $\kappa$ induces isomorphisms on homotopy groups and is therefore an equivalence.
Since $\ell$ is a mod-$p$ homology equivalence, completion also gives a zigzag of homotopy equivalences
\[
 BG_p^\wedge\xrightarrow[\simeq]{\ell_p^\wedge}
 L_p^\wedge\xleftarrow[\simeq]{\kappa}L_{(p)}
\]
\cite[Chapter~I, Section~5]{BK}.
By naturality of completion and the factorization involving $\kappa$, this is a zigzag under $BG$.

The map $S^2\vee L\to S^2\vee L_{(p)}$ is a homology equivalence with coefficients in $\Z_{(p)}$ between simply connected spaces, so the preceding equivalence and Theorem~\ref{loc:main} give \eqref{loc:completed-group-formula}.
If $p\nmid |G|$, transfer and the cofiber exact sequence give $\widetilde H_*(L;\Z_{(p)})=0$ and hence $L_{(p)}\simeq *$.
\end{proof}

\begin{corollary}[An arithmetic criterion for nontriviality]
\label{loc:arithmetic-nontriviality}
Let $X$ be a finite connected quandle, put $G=\Inn(X)$ and $H=\operatorname{Stab}_G(x_0)$, and assume $p\nmid |H|$.
Then
\begin{equation}\label{loc:arithmetic-triviality-equivalences}
 L_{(p)}\simeq *
 \quad\Longleftrightarrow\quad BG_p^\wedge\simeq *
 \quad\Longleftrightarrow\quad p\nmid |G|
 \quad\Longleftrightarrow\quad p\nmid |X|.
\end{equation}
Consequently, among the good primes, $L_{(p)}$ is
noncontractible exactly when $p$ divides $|X|$.
\end{corollary}

\begin{proof}
The first equivalence follows from Proposition~\ref{loc:completion}.
If $p\nmid |G|$, then $L_{(p)}\simeq *$ by Proposition~\ref{loc:completion}.
Conversely, if $p\mid |G|$, nonvanishing of the mod-$p$ cohomology of finite groups \cite[Proposition~3.3 and its proof]{Kahn1999} gives $\widetilde H^*(BG;\F_p)\ne0$.
By Lemma~\ref{loc:L-properties}, $BG\to L$ is a mod-$p$ cohomology equivalence, so $L_{(p)}$ is noncontractible.
The last equivalence follows from $|G|=|H||X|$ and $p\nmid|H|$.
\end{proof}

\subsection{Classification of local monoids and recovery of the completion}

At a good prime for a finite connected quandle, the converse to Proposition~\ref{loc:completion} holds: the homotopy type of $BG_p^\wedge$ can be recovered from the unmarked $A_\infty$ type of the covering monoid.

\begin{lemma}[Sphere cancellation]
\label{ainf:sphere-cancellation}
Let $V,V'$ be simply connected $p$-local CW spaces whose second integral homology groups are finite $p$-groups.
Then
\[
 S^2_{(p)}\vee V\simeq S^2_{(p)}\vee V'
 \quad\Longrightarrow\quad V\simeq V'.
\]
Moreover, if the image of $e\in\pi_2(S^2_{(p)}\vee V)$ generates the quotient by torsion as a $\Z_{(p)}$-module, then the homotopy cofiber of a based continuous map $S^2_{(p)}\to S^2_{(p)}\vee V$ representing $e$ is equivalent to $V$.
\end{lemma}

\begin{proof}
Put $R=\Z_{(p)}$.
The homological characterization shows that these wedges are $p$-local, and the Hurewicz theorem gives $\pi_2(S^2_{(p)}\vee V)=R\iota\oplus\pi_2(V)$, with torsion second factor.
Thus a homotopy equivalence $f$ sends $\iota$ to $a\iota'+t$, where $a\in R^\times$ and $t\in\pi_2(V')$.
The self-map $E$ that is the identity on $V'$ and sends $\iota'$ to $a^{-1}(\iota'-t)$ induces an isomorphism in integral homology, and hence is an equivalence by the simply connected Whitehead theorem.
Since $Ef$ preserves the sphere inclusion up to homotopy, taking cofibers gives $V\simeq V'$.
Applying the same shear to $e=a\iota+t$ proves the final assertion.
\end{proof}

\begin{theorem}[Classification of unmarked local monoids]
\label{ainf:classification}\label{ainf:marking}
Let $X,X'$ be finite connected quandles with inner automorphism groups $G,G'$ and point stabilizers $H,H'$.
If $p\nmid |H||H'|$, then
\begin{equation}\label{ainf:classification-equivalence}
 (\cM_X)_{(p)}\simeq_{A_\infty}(\cM_{X'})_{(p)}
 \quad\Longleftrightarrow\quad
 BG^\wedge_p\simeq BG'^\wedge_p.
\end{equation}
The space $BG^\wedge_p$ can be recovered as the homotopy cofiber of any based continuous map $S^2_{(p)}\to B((\cM_X)_{(p)})$ representing a $\Z_{(p)}$-generator of $\pi_2(B((\cM_X)_{(p)}))/\mathrm{torsion}$.
This recovers the homotopy type, not a particular equivalence or marking.
\end{theorem}

\begin{proof}
By Proposition~\ref{loc:completion} and Lemma~\ref{loc:L-properties}, $BG_p^\wedge\simeq L_{(p)}$ is simply connected, and its second integral homology is a finite $p$-group.
Proposition~\ref{loc:completion} gives the reverse implication in \eqref{ainf:classification-equivalence}.
For the forward implication, delooping an $A_\infty$ equivalence gives
$S^2_{(p)}\vee BG_p^\wedge\simeq S^2_{(p)}\vee BG'^\wedge_p$.
Lemma~\ref{ainf:sphere-cancellation} then gives the forward implication and the recovery statement.
\end{proof}

The recovered homotopy type of the completion does not determine the group $G$; recovering the rack space requires the marking $j$ of Theorem~\ref{int:main}.

\begin{remark}\label{loc:localization-not-completion}
The outer localization in \eqref{loc:completed-group-formula} is essential.
The wedge summand $S^2$ contributes $\Z_{(p)}$ to $\pi_2(Y_{(p)})$, whereas completion changes this summand to $\Z_p$.
Only the rationally acyclic term $L$ is replaced by $BG_p^\wedge$.
\end{remark}

\section{Fundamental groups, universal covers, and rationalization}
\label{sec:integral-consequences}\label{rat:section}

We apply the integral and localized models to compute $\pi_1(\cM)$, describe the universal cover of $\mathcal R_X$, and determine the rational homotopy type and multiplication of $\cM$.
We retain the generating transitive augmentation and the notation of Theorem~\ref{int:main}; finiteness assumptions will be stated where needed.
Write $H_n(G)=H_n(BG;\Z)$.

\subsection{Torus classes and the fundamental group}\label{low:torus-subsection}

We compute $\pi_1(\cM)$ from $H_2(G)$ and commuting torus classes.

Using the group homomorphism $b_h:H\times\Z\to H$ in \eqref{int:two-maps}, define a homomorphism of abelian groups by
\begin{equation}\label{low:torus-map}
 \tau_h:H_1(H)\longrightarrow H_2(G),\qquad
 \tau_h(v)=(Bi)_*(Bb_h)_*(v\times[\Z]).
\end{equation}
Here $[\Z]\in H_1(\Z)$ is the standard generator, and the cross product identifies $H_1(H)\otimes H_1(\Z)$ with a K\"unneth direct summand of $H_2(H\times\Z)$.
For $u\in H$, the class $\tau_h([u])$ is represented in the inhomogeneous bar complex by
\begin{equation}\label{low:bar-cycle}
 [i(u)\mid h]-[h\mid i(u)].
\end{equation}
This is a cycle because $i(u)$ and $h$ commute.

\begin{proposition}\label{low:fundamental-group}
Put
\[
 D=\ker\bigl(\Z\longrightarrow G_{\mathrm{ab}},\quad 1\longmapsto[h]\bigr).
\]
There is an exact sequence of abelian groups
\begin{equation}\label{low:exact-sequence}
 0\longrightarrow H_2(G)/\operatorname{im}\tau_h
 \longrightarrow \pi_1(\cM)
 \longrightarrow D\longrightarrow0,
\end{equation}
which splits noncanonically.
The group $G_{\mathrm{ab}}$ is cyclic, generated by $[h]$.
If $[h]$ has finite order $q$, then $D=q\Z$ and there is a noncanonical isomorphism
\begin{equation}\label{low:fundamental-splitting}
 \pi_1(\cM)\cong\Z\oplus
       H_2(G)/\operatorname{im}\tau_h.
\end{equation}
If $[h]$ has infinite order, then $D=0$.
When $G$ is finite, the group $H_2(G)/\operatorname{im}\tau_h$ is finite.
\end{proposition}

\begin{proof}
Since the conjugates of $h$ generate $G$, the element $[h]$ generates $G_{\mathrm{ab}}$.
Apply the K\"unneth decomposition to the Mayer--Vietoris sequence of \eqref{int:Y-definition} and cancel the direct summands $H_2(H)$ and $H_1(H)$ for which the second component is $-\id$.
This gives an exact sequence of abelian groups
\[
 H_1(H)\xrightarrow{\tau_h}H_2(G)\longrightarrow H_2(Y)
 \xrightarrow{\delta}\Z\xrightarrow{1\mapsto[h]}G_{\mathrm{ab}}
 \longrightarrow0.
\]
Now \eqref{low:exact-sequence} follows from $\operatorname{im}\delta=D$ and the identification $H_2(Y)\cong\pi_2(Y)\cong\pi_1(\cM)$ given by the Hurewicz theorem and Theorem~\ref{int:main}.
The sequence splits because every subgroup of $\Z$ is free.
The descriptions of $D$ follow from its definition.
If $G$ is finite, then $H_2(G)$ is finite.
The collapse of the defining span to $*\leftarrow B\Z\to *$ gives a map $Y\to S^2$ inducing $\delta$, up to orientation, by naturality of Mayer--Vietoris; its image is $D$.
\end{proof}

We now relate this computation to associated groups and reduced Schur multipliers.

For connected $X$, the presentation \eqref{pre:associated-group} identifies $\As(X)_{\mathrm{ab}}$ with $\Z$ via $\varepsilon:\As(X)\to\Z$, $e_x\mapsto1$.
Let $\psi:\As(X)\to\Inn(X)$ be the canonical epimorphism and put
\[
 \overline G=\operatorname{im}(\varepsilon,\psi),\qquad
 \overline h=(1,S_{x_0}).
\]
The action through the second projection and the map $x\mapsto(1,S_x)$ define a generating transitive augmentation.
The first projection induces $\overline G_{\mathrm{ab}}\cong\Z$, so $D=0$; Proposition~\ref{low:fundamental-group} and \eqref{pre:fundamental-cover} therefore give
\[
 \ker(\varepsilon,\psi)\cong
 H_2(\overline G)/\operatorname{im}\tau_{\overline h},
\]
where $\tau_{\overline h}$ uses $\operatorname{Stab}_{\overline G}(x_0)$.
By transitivity and conjugation invariance, these torus relations agree up to sign with those of \cite[Theorem~3.1]{IshikawaAssociated2024}, giving a geometric proof of its kernel formula for connected quandles.
Ishikawa also determines the central extension class and treats nonconnected racks.

For a finite faithful connected quandle with its canonical augmentation, $X$ is identified with the conjugacy class $c$ of $h$ and $H=C_G(h)$.
Thus $H_2(G)/\operatorname{im}\tau_h$ is the reduced Schur multiplier $H_2(G,c)$ \cite[Section~4B]{RobertsVenkatesh2015}, which is isomorphic to $\ker\psi\cap[\As(X),\As(X)]$ \cite[Proposition~6.9]{Loenne2019}.
For a nonfaithful quandle, the torus relations use the stabilizer $H$, which need not equal $C_G(h)$.

\subsection{Universal covers}\label{low:covers-subsection}

Let $Y\langle2\rangle$ be the $2$-connected cover of $Y$, namely the homotopy fiber of a based map $Y\to K(\pi_2Y,2)$ inducing the identity on $\pi_2$.

\begin{corollary}\label{low:universal-cover}
There are ordinary homotopy equivalences
\begin{equation}\label{low:universal-model}
 \widetilde{\mathcal{R}_X}\simeq\widetilde \cM\simeq\Omega(Y\langle2\rangle).
\end{equation}
If $\widetilde \cM$ is equipped with the multiplication lifted from $\cM$, the second equivalence may be chosen to be an $A_\infty$ equivalence.
\end{corollary}

\begin{proof}
The covering $\cM\to \mathcal{R}_X$ identifies the universal covers.
Since $\Omega(Y\langle2\rangle)$ models the simply connected cover of $\Omega Y$, Theorem~\ref{int:main} gives \eqref{low:universal-model}.
The multiplication on $\cM$ lifts uniquely to the based universal cover, preserving associativity and the unit.
The bar--loop comparison shows that
$B\widetilde\cM\to B\cM$ is a $2$-connected cover.
Together with $B\cM\simeq Y$, this gives
$B\widetilde\cM\simeq Y\langle2\rangle$.
Looping yields the required $A_\infty$ equivalence.
\end{proof}

Under a vanishing assumption on second homology, the localized universal cover has the following description.

\begin{corollary}\label{loc:universal-cover-explicit}
For a generating transitive augmentation as in Section~\ref{int:section}, suppose that $G$ is finite, $p\nmid |H|$, and $H_2(G;\Z)\otimes\Z_{(p)}=0$.
Then there is an ordinary homotopy equivalence
\begin{equation}\label{loc:universal-cover-formula}
 (\widetilde{\mathcal{R}_X})_{(p)}\simeq
 \Omega\bigl((S^3\vee(S^1_+\wedge L))_{(p)}\bigr).
\end{equation}
If $\widetilde{\mathcal{R}_X}$ is replaced by $\widetilde\cM$ equipped with its lifted multiplication, the equivalence can be chosen to be an $A_\infty$ equivalence.
Here $S^1_+$ is the circle with a disjoint basepoint adjoined, and
\[
 S^1_+\wedge L=(S^1\times L)/(S^1\times\{*\}).
\]
\end{corollary}

\begin{proof}
Put $W=S^2\vee L$, and let $c:W\to K(\Z,2)$ be a map representing a generator on $S^2$ and constant on $L$.
Its homotopy fiber $E$ is obtained by gluing the Hopf bundle and a trivial circle bundle along their common fiber.
Using the fiber inclusion $S^1\to S^3$ and the subcomplex inclusion $S^1\to S^1\times L$ gives the homotopy equivalences
\[
 E\simeq\hocolim[S^3\longleftarrow S^1\longrightarrow S^1\times L]
   \simeq S^3\vee(S^1_+\wedge L).
\]
The first arrow is null-homotopic, and the second is the inclusion of the subcomplex $S^1\times\{*\}$.
The fiber sequence shows that $E$ is simply connected, $\pi_2(E)\cong\pi_2(L)$, and $\pi_n(E)\cong\pi_n(W)$ for $n\ge3$.
By the Hurewicz theorem and Lemma~\ref{loc:L-properties},
\[
 \pi_2(E)\otimes\Z_{(p)}
   \cong H_2(G;\Z)\otimes\Z_{(p)}=0
\]
so $E_{(p)}\to W_{(p)}$ is the $2$-connected cover.
The comparison $\widetilde \cM\to \cM_{(p)}\simeq\Omega W_{(p)}$ of Theorem~\ref{loc:main} lifts to $\Omega E_{(p)}$ and localizes homotopy groups in degrees at least two.
Thus $(\widetilde \cM)_{(p)}\simeq\Omega E_{(p)}$, and we obtain an $A_\infty$ equivalence by looping the comparison between the $2$-connected covers of the deloopings.
Corollary~\ref{low:universal-cover} gives the same model for $(\widetilde{\mathcal{R}_X})_{(p)}$.
\end{proof}

Without the vanishing assumption, we have $\pi_1(\Omega E_{(p)})=H_2(G;\Z)_{(p)}$.

\begin{remark}[Rationalization and multiplication]\label{rat:model}
Let $X$ be a finite connected quandle.
Put $G=\Inn(X)$ and $\cM=B(G;X)$, and let $Y$ be the space in \eqref{int:Y-definition}.
Since $L_{\Q}\simeq *$ by Lemma~\ref{loc:L-properties}, Theorem~\ref{loc:main} gives, respectively, a homotopy equivalence and an $A_\infty$ equivalence
\begin{equation}\label{rat:loop-equivalences}
 Y_{\mathbb Q}\simeq S^2_{\mathbb Q},\qquad
 \cM_{\mathbb Q}\simeq_{A_\infty}\Omega S^2_{\mathbb Q}.
\end{equation}
In particular, $\pi_j(\cM)\otimes\Q=0$ for $j\ge3$.

Under \eqref{rat:loop-equivalences}, let $\alpha\in\pi_1(\cM_{\Q})$ and $\beta\in\pi_2(\cM_{\Q})$ correspond to the adjoints of the rationalizations of $\iota_2=[\id_{S^2}]\in\pi_2(S^2)$ and the Hopf class $\eta\in\pi_3(S^2)$, respectively.
The integral identity $[\iota_2,\iota_2]=\pm2\eta$
\cite[Section~4.J, immediately before Corollary~4J.4]{Hatcher}
and the Whitehead--Samelson adjunction give
\begin{equation}\label{rat:samelson-square}
 [\alpha,\alpha]_{\mathrm{Sam}}=\pm2\beta\ne0.
\end{equation}
\end{remark}

\section{From the covering monoid to rack and quandle homology}
\label{hom:section}

Throughout this section, $X$ is a finite connected quandle, equipped with its canonical augmentation unless otherwise stated.
We use the notation $G,H,\cM,\mathcal{N}_X,L$ introduced above.
Propositions~\ref{hom:covering} and~\ref{hom:shift} also apply to generating transitive augmentations into finite groups $G$.
Coefficients are constant, and we write $H_n^R(X;\Lambda)=H_n(\mathcal{R}_X;\Lambda)$ and $H_R^n(X;\Lambda)=H^n(\mathcal{R}_X;\Lambda)$, including $H_0^R=H_R^0=\Lambda$.
For chain complexes, our shift convention is $D[m]_n=D_{n-m}$ with differential $d_{D[m]}=(-1)^m d_D$.

We convert the homology of $\Omega L$ into rack and quandle homology, retaining torsion orders for the computations in Sections~\ref{sec:cyclic}--\ref{ex:nonalex-section}.
The final subsection determines the ordinary rack cup product over a field.

\subsection{Transfer and the degree shift}

\begin{proposition}\label{hom:covering}
Let $q:\cM\to\mathcal{N}_X$ be the principal $H$-covering in
\eqref{pre:intermediate-cover}.
Every deck transformation is homotopic to the identity.
Consequently, for a commutative ring $\Lambda$ in which $|H|$ is invertible, the induced $\Lambda$-module homomorphisms
\[
 q_*\colon H_n(\cM;\Lambda)\xrightarrow{\ \cong\ }H_n(\mathcal{N}_X;\Lambda),
 \qquad
 q^*\colon H^n(\mathcal{N}_X;\Lambda)\xrightarrow{\ \cong\ }H^n(\cM;\Lambda)
\]
are isomorphisms for all $n\geq0$.
The same holds with coefficients in any constant $\Lambda$-module.
\end{proposition}

\begin{proof}
The identification \eqref{pre:intermediate-cover} exhibits $q$ as a principal $H$-covering.
Since $\cM$ is connected, choose a path $\gamma_h$ from its unit to $h\in H$.
Then $(s,m)\mapsto\gamma_h(s)m$ gives a homotopy from the identity to the deck transformation $h$.
Thus the transfer $\operatorname{tr}$, a $\Lambda$-module homomorphism, satisfies
\[
 q_*\operatorname{tr}=|H|\,\id,
 \qquad
 \operatorname{tr}q_*=
 \sum_{h\in H}(h\cdot-)_*=|H|\,\id.
\]
It follows that $|H|^{-1}\operatorname{tr}$ is the inverse of $q_*$.
The same transfer argument applies to cohomology and to any constant $\Lambda$-module.
\end{proof}

This is Clauwens' transfer argument \cite[Proposition~4.2]{Clauwens}.
The following degree shift is the chain identification of \cite[\S2.3]{Clauwens}; its integral homological form is also recorded in \cite[Remark~4.1]{NosakaSurfaces2013}.

\begin{proposition}\label{hom:shift}
For every commutative coefficient ring $\Lambda$ and every $n\geq0$, there are $\Lambda$-module isomorphisms
\begin{equation}\label{hom:shift-groups}
 H_{n+1}^R(X;\Lambda)\cong H_n(\mathcal{N}_X;\Lambda),
 \qquad
 H_R^{n+1}(X;\Lambda)\cong H^n(\mathcal{N}_X;\Lambda)
\end{equation}
natural in the coefficients.
In particular, if $p\nmid|H|$, then
\begin{equation}\label{hom:shift-cover}
 H_{n+1}^R(X;\Z_{(p)})\cong H_n(\cM;\Z_{(p)}),
 \qquad
 H_R^{n+1}(X;\Z_{(p)})\cong H^n(\cM;\Z_{(p)}).
\end{equation}
The corresponding statements hold with coefficients in any $\Z_{(p)}$-module, including any field of characteristic $p$.
\end{proposition}

\begin{proof}
Reindex the positive-degree integral rack complex as
$D_n=C_{n+1}^R(X;\Z)$ for $n\ge0$, with $D_n=0$ for $n<0$,
retaining the rack differential.
The cited chain isomorphism, with the conventions of \eqref{pre:faces0}--\eqref{pre:rack-chain}, is
\begin{equation}\label{hom:psi}
 \psi_n\colon D_n\longrightarrow C_n(X;X;\Z),\qquad
 \psi_n(x_0',x_1,\ldots,x_n)
   =(-1)^n(x_0';x_1,\ldots,x_n).
\end{equation}
Since $\partial_1=0$, tensoring with $\Lambda$ and applying $\operatorname{Hom}_{\Z}(-,\Lambda)$ give \eqref{hom:shift-groups} for all $n\ge0$, naturally in the coefficients.
The same argument applies to any constant coefficient module.
The remaining assertions follow from Proposition~\ref{hom:covering}.
\end{proof}

For finite connected Alexander quandles whose defining automorphism has order coprime to the quandle order, the integral comparison $H_n(\cM;\Z)\cong H_{n+1}^R(X;\Z)$ in the range $n\ge1$ already appears in \cite[Lemma~5.9]{NosakaSurfaces2013}.

\subsection{Loop homology of a wedge}

We recall the algebraic fact needed to use the localized homotopy model.
For augmented algebras, $*$ denotes the unital free product with the copies of the coefficient ring identified.
We write $T_\Lambda(V)=\bigoplus_{m\ge0}V^{\otimes_\Lambda m}$ for the unital tensor algebra on a graded $\Lambda$-module $V$, and abbreviate $T_\Lambda(\Lambda u)$ to $T_\Lambda(u)$.

\begin{lemma}\label{hom:loop-wedge}
Let $U$ and $V$ be based simply connected CW spaces, and let $k$ be a field.
Looping the wedge inclusions induces an isomorphism of graded Hopf algebras
\begin{equation}\label{hom:loop-wedge-iso}
 H_*(\Omega U;k)*H_*(\Omega V;k)
    \xrightarrow{\ \cong\ }H_*(\Omega(U\vee V);k).
\end{equation}
In particular, for a based simply connected CW space $L$,
\begin{equation}\label{hom:sphere-free-product}
 H_*(\Omega(S^2\vee L);k)
       \cong T_k(u)*H_*(\Omega L;k),\qquad |u|=1,
\end{equation}
and $u$ is primitive.
No $2$-connectedness assumption on $L$ is needed.
\end{lemma}

\begin{proof}
The first assertion is the classical coproduct formula for loop-space homology \cite[p.~245, Remark]{ClarkSmith1968}.
The second follows from the Bott--Samelson theorem \cite{BottSamelson1953}, which gives $H_*(\Omega S^2;k)\cong T_k(u)$ with $|u|=1$.
Since this Hopf algebra is connected, $u$ is primitive for degree reasons.
\end{proof}

The loop model also determines the Pontryagin Hopf algebra.

\begin{theorem}[Pontryagin Hopf algebra]\label{hom:series}
Let $p$ be a prime with $p\nmid|H|$, let $k$ be a field of characteristic $p$, and put $L=\Cofib(BH\to BG)$.
There is an isomorphism of graded Hopf algebras
\begin{equation}\label{hom:monoid-algebra}
 H_*(\cM;k)\cong T_k(u)*H_*(\Omega L;k),\qquad |u|=1,
 \qquad\Delta u=u\otimes1+1\otimes u.
\end{equation}
\end{theorem}

\begin{proof}
Theorem~\ref{loc:main} and Lemma~\ref{hom:loop-wedge} give \eqref{hom:monoid-algebra}, since localization of connected nilpotent loop spaces preserves the homology Hopf algebra with coefficients in $k$.
\end{proof}

\subsection{Quandle normalization and degenerate summands}

This subsection applies to any finite quandle, without assumptions on $G$, $H$, or the prime.

We follow \cite[\S3]{CJKLS2003} for the definition of the quandle complex.
For $n\ge2$, let $C_n^D(X;\Lambda)$ be the submodule of $C_n^R(X;\Lambda)$ generated by tuples satisfying $x_i=x_{i+1}$ for some $i$.
Put $C_n^D=0$ for $n\le1$.
Quandle idempotency ensures that these modules form a subcomplex.
Write $H_n^D(X;\Lambda)=H_n(C_*^D(X;\Lambda))$ and define
\[
 C_n^Q(X;\Lambda)=C_n^R(X;\Lambda)/C_n^D(X;\Lambda)
 \quad(n\ge1).
\]
The differential out of degree one in this positive-degree complex is zero.
We denote its positive-degree homology and cohomology by $H_n^Q$ and $H_Q^n$, respectively.
We do not adopt a convention for $H_0^Q$.
All quandle Poincar\'e series below start in degree one, whereas the rack Poincar\'e series includes $H_0^R$.

\begin{proposition}\label{hom:degenerate}
Let $\Lambda$ be a principal ideal domain and write
\[
 R_i=H_i^R(X;\Lambda)\quad(i\geq0),\qquad
 Q_j=H_j^Q(X;\Lambda)\quad(j\geq1),\qquad R_0=\Lambda.
\]
For $n\geq1$, there is a splitting $H_n^R(X;\Lambda)\cong Q_n\oplus
H_n^D(X;\Lambda)$.
Moreover, the following sequence of $\Lambda$-module homomorphisms is split exact, although its splitting is generally nonnatural:
\begin{equation}\label{hom:degenerate-kunneth}
 0\to
 \bigoplus_{\substack{i+j=n-1\\i\ge0,\ j\ge1}}R_i\otimes_\Lambda Q_j
 \to H_n^D(X;\Lambda)\to
 \bigoplus_{\substack{i+j=n-2\\i\ge0,\ j\ge1}}
 \operatorname{Tor}_1^\Lambda(R_i,Q_j)\to0.
\end{equation}
An empty direct sum is understood to be zero.
\end{proposition}

\begin{proof}
The chain-level splitting and tensor-product description are given in \cite[Theorems~2.7 and~7.1]{PrzytyckiPutyra2016}.
In their conventions, indicated by the superscript $\mathrm{PP}$, tuples in $X^{a+1}$ have degree $a$.
Write $C^{N,\mathrm{PP}}$ and $C^{D,\mathrm{PP}}$ for their normalized and degenerate complexes, and $\widehat C^{\mathrm{PP}}$ for the rack complex extended by $\Lambda$ in degree $-1$.
For trivial coefficients, the added differential $\Lambda[X]\to\Lambda$ sends $x$ to $1-1=0$.
Their differential in nonnegative degrees is the negative of ours, so reindexing and multiplying by $(-1)^a$ in degree $a$ gives
\[
 H_a(\widehat C^{\mathrm{PP}})=R_{a+1}\quad(a\ge-1),\qquad
 H_b(C^{N,\mathrm{PP}})=Q_{b+1}\quad(b\ge0).
\]
Apply the cited chain isomorphism $\widehat C^{\mathrm{PP}}[2]\otimes_\Lambda C^{N,\mathrm{PP}}
\cong C^{D,\mathrm{PP}}$ and then the split K\"unneth theorem for free complexes.
The tensor terms are indexed by $i+j=n-1$, and the Tor terms by $i+j=n-2$.
This gives \eqref{hom:degenerate-kunneth}.
\end{proof}

\begin{corollary}\label{hom:quandle-series}\label{hom:rack-determines-quandle}
Over a principal ideal domain $\Lambda$, the rack homology groups of a finite quandle determine the module isomorphism classes of its quandle homology groups in all positive degrees.
In particular, if $\Lambda=k$ is a field, put $r_n=\dim_k H_n^R(X;k)$, $q_n=\dim_k H_n^Q(X;k)$ for $n\ge1$, and $r_0=1$.
Then
\begin{equation}\label{hom:quandle-recurrence}
 q_n=r_n-\sum_{j=1}^{n-1}r_{n-1-j}q_j\qquad(n\ge1).
\end{equation}
The positive-degree quandle series $Q_{X,+}(t)=\sum_{n\ge1}q_nt^n$ and the series $R_X(t)=\sum_{n\ge0}r_nt^n$ satisfy
\begin{equation}\label{hom:rack-quandle-series}
 Q_{X,+}(t)=\frac{R_X(t)-1}{1+tR_X(t)}.
\end{equation}
We also have $q_n=\dim_k H_Q^n(X;k)$.
\end{corollary}

\begin{proof}
Every term $Q_j$ in \eqref{hom:degenerate-kunneth} has $j<n$.
The structure theorem for finitely generated $\Lambda$-modules and the splitting in Proposition~\ref{hom:degenerate} determine $Q_n$ by direct-sum cancellation from $R_n$ and the already determined groups $Q_j$.
Over a field, Tor vanishes; taking dimensions gives \eqref{hom:quandle-recurrence}, and summing gives \eqref{hom:rack-quandle-series}.
The dimensions of the cohomology groups follow from the universal coefficient theorem.
\end{proof}

\subsection{Additive formulas and torsion orders}

We first give a common formula for the additive groups, then specialize it to dimensions over a field and to counts of cyclic summands of each order.
All series are interpreted as formal power series.

For a principal ideal domain $\Lambda$, we record separately the rank of the free part and the cyclic direct summands $\Lambda/(\pi^e)$ corresponding to each prime power $\pi^e$.
Let $\mathcal G_\Lambda$ be the $t$-adic completion of the commutative $\Z[t]$-algebra whose underlying $\Z[t]$-module has basis $1$ and $\epsilon_{\pi,e}$.
Here $\pi$ runs through representatives of the prime elements of $\Lambda$ up to associates, and $e\geq1$.
Define multiplication by
\begin{equation}\label{hom:pid-bookkeeping-product}
 \epsilon_{\pi,e}\epsilon_{\pi,f}
     =(1+t)\epsilon_{\pi,\min(e,f)},\qquad
 \epsilon_{\pi,e}\epsilon_{\sigma,f}=0\quad(\pi\not\sim\sigma).
\end{equation}
For the series considered below, each coefficient involves only finitely many basis symbols.
For a finitely generated $\Lambda$-module $A$, define $[A]$ as the free rank of $A$ times $1$, plus one $\epsilon_{\pi,e}$ for each cyclic summand $\Lambda/(\pi^e)$.
By the structure theorem, this notation is independent of the chosen decomposition.
For a field $k$, this convention gives $\mathcal G_k=\Z[[t]]$ and $[A]=\dim_k A$.
Formula~\eqref{hom:pid-bookkeeping-product} records both tensor and Tor terms:
\[
 [A][B]=[A\otimes_\Lambda B]
          +t[\operatorname{Tor}_1^\Lambda(A,B)].
\]
The term multiplied by $t$ records the Tor contribution, which is absent over a field.

\begin{lemma}[Additive homology formulas]\label{hom:pid-words}
Let $\Lambda$ be a principal ideal domain and $L$ a based simply connected CW space, and assume that $H_n(\Omega L;\Lambda)$ is finitely generated for every $n$.
Put
\[
 \kappa(t)=\sum_{n\geq0}[H_n(\Omega L;\Lambda)]t^n,
 \qquad
 \mu(t)=\sum_{n\geq0}[H_n(\Omega(S^2\vee L);\Lambda)]t^n.
\]
Then, in $\mathcal G_\Lambda$,
\begin{equation}\label{hom:pid-loop-series}
 \mu(t)=\sum_{m\geq0}t^m\kappa(t)^{m+1}
       =\frac{\kappa(t)}{1-t\kappa(t)}.
\end{equation}
For the quandle data of this section, take $L=\Cofib(BH\to BG)$, a prime $p\nmid|H|$, and $\Lambda=\Z_{(p)}$ or a field of characteristic $p$.
The finite-generation hypothesis then holds, and the series
\[
 \rho(t)=\sum_{n\geq0}[H_n^R(X;\Lambda)]t^n,
 \qquad
 \chi(t)=\sum_{n\geq1}[H_n^Q(X;\Lambda)]t^n
\]
satisfy
\begin{equation}\label{hom:pid-series}
 \rho(t)=\frac{1}{1-t\kappa(t)},\qquad
 \chi(t)=\frac{t\kappa(t)}{1+t-t\kappa(t)}.
\end{equation}
These formulas record module isomorphism classes, including torsion orders; they make no assertion about cup products or homology algebras over a principal ideal domain.
\end{lemma}

\begin{proof}
Choose a $1$-reduced simplicial model $K$ of $L$, meaning that it has one vertex and no nondegenerate $1$-simplices.
Put $\mathcal C=\Omega C_*(K;\Lambda)$, where $C_*$ denotes normalized chains and $\Omega$ the cobar construction; then $\mathcal C_0=\Lambda$.
For $S=\Delta[2]/\partial\Delta[2]$, the reduced chains of $K\vee S$ and their reduced Alexander--Whitney coproduct split as direct sums, giving an isomorphism of differential graded algebras
\begin{equation}\label{hom:cobar-wedge}
 \Omega C_*(K\vee S;\Lambda)\cong\mathcal C*T_\Lambda(u),
 \qquad |u|=1,\quad du=0.
\end{equation}
The natural integral Szczarba comparison with chains on the Kan loop group \cite[Theorem~3.1 and the following paragraph]{HPST2006} has an acyclic bounded-below cone of free abelian groups, hence a contractible cone.
It remains a quasi-isomorphism after tensoring with $\Lambda$, so \eqref{hom:cobar-wedge} models loop chains compatibly with the inclusions.

The differential preserves the number $m$ of occurrences of $u$.
Separating the words $c_0uc_1u\cdots uc_m$, with each $c_i$ allowed to be the unit, gives a chain isomorphism
\begin{equation}\label{hom:pid-word-complex}
 \mathcal C*T_\Lambda(u)
    \cong\bigoplus_{m\geq0}
          \mathcal C^{\otimes_\Lambda(m+1)}[m].
\end{equation}
On homogeneous factors, the isomorphism from the $m$th summand is
\[
 c_0\otimes\cdots\otimes c_m\longmapsto
 (-1)^{\sum_{i=0}^m(m-i)|c_i|}c_0uc_1u\cdots uc_m.
\]
The split K\"unneth theorem, recorded by \eqref{hom:pid-bookkeeping-product}, gives \eqref{hom:pid-loop-series}.
Only $m\le n$ contributes in degree $n$, so the resulting homology groups are finitely generated.

For the quandle application, Theorem~\ref{loc:main} and localization identify $H_*(\cM;\Lambda)$ with $H_*(\Omega(S^2\vee L);\Lambda)$.
The looped wedge retraction makes $H_n(\Omega L;\Lambda)$ a direct summand of $H_n(\cM;\Lambda)$, which is finitely generated since $\cM$ has finitely many cells in each degree.
Proposition~\ref{hom:shift} and the split K\"unneth formula \eqref{hom:degenerate-kunneth} give, respectively,
\[
 \rho=1+t\mu,\qquad \rho-1=\chi+t\rho\chi.
\]
The product in $\mathcal G_\Lambda$ includes the Tor terms; cancellation is justified by Corollary~\ref{hom:rack-determines-quandle}.
Solving these identities gives \eqref{hom:pid-series}; the denominators have constant term $1$ and hence are invertible in $\mathcal G_\Lambda$.
\end{proof}

For the quandle data above and a field $k$ of characteristic $p$ with $p\nmid|H|$, write
\[
 A_X(t)=\sum_{n\geq0}\dim_k H_n(\Omega L;k)t^n,
 \qquad
 P_\cM(t)=\sum_{n\geq0}\dim_k H_n(\cM;k)t^n.
\]
Lemma~\ref{hom:pid-words} ensures finite-dimensionality.
With $R_X$ and $Q_{X,+}$ as in Corollary~\ref{hom:quandle-series}, formulas \eqref{hom:pid-loop-series} and \eqref{hom:pid-series} apply with $\Lambda=k$ and the identifications
\[
 (\kappa,\mu,\rho,\chi)=(A_X,P_\cM,R_X,Q_{X,+}).
\]
The universal coefficient theorem gives the same dimension series for cohomology.
Clauwens obtained the corresponding rack recurrence under a Leray--Hirsch hypothesis and applied it to dihedral quandles \cite[Section~4.3, Remark~4.11]{Clauwens}.

\begin{corollary}[Multiplicities of cyclic summands]\label{hom:exponent-profiles}
Retain the quandle data and the good prime $p$ of Lemma~\ref{hom:pid-words}.
Put $R=\Z_{(p)}$ and $K=\Omega(L_{(p)})$.
For an integer $e\ge1$ and a finite abelian $p$-group $A$, let $\nu_e(A)$ be the number of cyclic direct summands $C_{p^f}$ with $f\ge e$.
The groups $H_n(K;R)$ for $n>0$ and $H_n^Q(X;R)$ for $n\ge2$ are finite abelian $p$-groups, and $H_1^Q(X;R)=R$.
The cumulative exponent series
\[
 K_{\ge e}(t)=\sum_{n\ge1}\nu_e(H_n(K;R))t^n,\qquad
 Q_{\ge e}(t)=\sum_{n\ge2}\nu_e(H_n^Q(X;R))t^n
\]
satisfy
\begin{equation}\label{hom:exponent-profile-series}
 Q_{\ge e}(t)=\frac{t(1+t)K_{\ge e}(t)}{1-t(1+t)K_{\ge e}(t)}.
\end{equation}
The series counting cyclic direct summands of order exactly $p^e$ in quandle homology is $Q_{\ge e}-Q_{\ge e+1}$.
Moreover, with coefficients in $\F_p$,
\begin{equation}\label{hom:exponent-profile-uct}
 A_X(t)=1+(1+t)K_{\ge1}(t),\qquad
 Q_{X,+}(t)=t+(1+t)Q_{\ge1}(t).
\end{equation}
\end{corollary}

\begin{proof}
Localization identifies $H_*(\Omega L;R)$ with $H_*(K;R)$.
By Lemma~\ref{hom:pid-words} and $L_{\Q}\simeq *$ from Lemma~\ref{loc:L-properties}, the positive-degree homology groups of $K$ are finite $p$-groups.
The rank specialization $\epsilon_{p,f}\mapsto0$ sends $\kappa$ to $1$ and, by \eqref{hom:pid-series}, $\chi$ to $t$.
Thus $H_n^Q(X;R)$ is finite for $n\ge2$, while connectedness gives $H_1^Q(X;R)=R$.

For each $e\ge1$, the rule $\epsilon_{p,f}\mapsto1+t$ for $f\ge e$ and $\epsilon_{p,f}\mapsto0$ otherwise, fixing $t$ and $1$, defines a continuous ring homomorphism to $\Z[[t]]$ by \eqref{hom:pid-bookkeeping-product}.
It sends $\kappa$ and $\chi$ to $1+(1+t)K_{\ge e}$ and $t+(1+t)Q_{\ge e}$, respectively.
Applying it to \eqref{hom:pid-series} gives \eqref{hom:exponent-profile-series}.
Successive differences count exact orders, and the universal coefficient theorem gives \eqref{hom:exponent-profile-uct}.
\end{proof}

\paragraph{Recovering integral homology.}\label{hom:alexander-integral-recovery}
For a finite connected Alexander quandle $X$ of order $N$, multiplication by $N$ annihilates $H_n^Q(X;\Z)$ for $n\ge2$ \cite[Corollary~6.2]{Nosaka2013}.
These finitely generated groups are therefore finite and are determined by their localizations at the primes dividing $N$.
In particular, when $N=p^a$, localization at $p$ determines the integral groups.
In degree one, connectedness gives $H_1^Q(X;\Z)\cong\Z$.

\subsection{The rack cup product}

We now use the marked fiber model to determine the ordinary cup product on $H_R^*(X;k)=H^*(\mathcal R_X;k)$, which agrees with the cubical rack cup product \cite[Sections~2.8 and~3]{Clauwens}.

\begin{proposition}[Rack cohomology as a shuffle algebra]\label{hom:rack-shuffle}
Let $p$ be a prime with $p\nmid|H|$, let $k$ be a field of characteristic $p$, and put $K=\Omega L$, where $L=\Cofib(BH\to BG)$.
There is a noncanonical isomorphism of graded unital algebras
\begin{equation}\label{hom:rack-shuffle-algebra}
 H_R^*(X;k)\cong\operatorname{Sh}_k\bigl(sH^*(K;k)\bigr).
\end{equation}
Here $s$ raises degrees by one, so $(sH^*(K;k))^n=H^{n-1}(K;k)$ for $n\ge1$, and $\operatorname{Sh}_k(V)$ is $\bigoplus_{m\ge0}V^{\otimes m}$ with the shuffle product and Koszul signs computed from the degrees in $V$.
Only the graded vector space $H^*(K;k)$ is used, including the suspension of $H^0(K;k)$ as a letter in degree one.
\end{proposition}

\begin{proof}
Put $W=S^2\vee L$, let $i:L\to W$ be the inclusion, and set $F=\hofib(i)$.
The marked comparisons in Theorems~\ref{int:main} and~\ref{loc:main} give a map of homotopy fibrations
\[
\begin{array}{ccccc}
 \Omega Y&\longrightarrow&\mathcal R_X&\longrightarrow&BG\\
 {\scriptstyle\Omega\Phi}\downarrow&&{\scriptstyle g}\downarrow&&{\scriptstyle\ell}\downarrow\\
 \Omega W&\longrightarrow&F&\longrightarrow&L.
\end{array}
\]
By Lemma~\ref{loc:L-properties} and localization of loop spaces, $\ell$ and $\Omega\Phi$ are $k$-homology equivalences.
Both rows are pullbacks of path fibrations over the simply connected spaces $Y$ and $W$, so their homology local systems are trivial.
The Serre spectral sequence comparison therefore shows that $g$ is a $k$-homology equivalence, giving an isomorphism $g^*:H^*(F;k)\to H_R^*(X;k)$ of cup-product algebras.

Let $r:W\to L$ be the pinch map and put $E=\hofib(r)$.
The standard description of the fiber of a wedge retraction \cite[Example~3.6 and Remark~3.7]{Theriault2024} gives
\[
 E\simeq (S^2\times K)/(\{*\}\times K)
       =S^2\wedge K_+=\Sigma^2K_+,
\]
where $K_+$ denotes $K$ with a disjoint basepoint.
Since $ri=\id_L$, the homotopy exact sequences and the splitting $r_*i_*=\id$ show that the natural map $\Omega E\to\Omega W\to F$ induces isomorphisms on all homotopy groups.
Both spaces are connected and have CW type, so $F\simeq\Omega\Sigma^2K_+$.

Set $V=\Sigma K_+$, which is connected.
The Bott--Samelson theorem \cite{BottSamelson1953} gives
\[
 H_*(\Omega\Sigma V;k)\cong
 T_k\bigl(\widetilde H_*(V;k)\bigr)
 =T_k\bigl(sH_*(K;k)\bigr).
\]
Because $V$ is a suspension, its reduced diagonal is nullhomotopic; hence the letters are primitive and the coproduct is the signed unshuffle coproduct.
By Lemma~\ref{hom:pid-words}, the letter space is finite dimensional in each degree and concentrated in positive degrees.
Thus each total degree has only finitely many word lengths, and taking degreewise duals gives $\operatorname{Sh}_k(sH^*(K;k))$, with no completion required.
Together with $g^*$, this proves the assertion.
\end{proof}

\section{Cyclic Alexander quandles and cyclic Sylow subgroups}\label{sec:cyclic}

We prove Theorem~\ref{intro:cyclic} and then extend the local computation to finite connected quandles $X$ at good primes $p$ for which a Sylow $p$-subgroup of $\Inn(X)$ is nontrivial and cyclic.
This gives the integral quandle homology of regular cyclic Alexander quandles.

We first record the augmentation data used in the Alexander examples.
Let $A$ be a finite abelian group and let $T\in\operatorname{Aut}(A)$ be such that $1-T$ is invertible.
Equip $X=A$ with the Alexander operation $x*y=Tx+(1-T)y$.
Since $(1-T)A=A$, the right translations generate all translations of $A$ together with $T$.
Thus $X$ is connected, and its inner automorphism group and the stabilizer of $0$ are
\begin{equation}\label{low:affine-data}
 G\cong A\rtimes\langle T\rangle,\qquad H\cong\langle T\rangle=C_d,
 \qquad d=\operatorname{ord}(T).
\end{equation}
Our right-action multiplication is opposite to the usual left affine multiplication; inversion identifies the two groups and sends $T$ to $T^{-1}$.
The quandle is called \emph{regular} if $\gcd(|A|,d)=1$.
Since $|H|=d$, Proposition~\ref{loc:completion} applies at every prime $p\nmid d$, in particular at every prime dividing $|A|$ in the regular case.

For the cyclic case, fix an odd prime $p$, an integer $r\ge1$, and a unit $\omega\in(\Z/p^r)^\times$ of order $d>1$ dividing $p-1$.
Put
\begin{equation}\label{cyc:definition}
 X_{p^r,\omega}=\Z/p^r,\qquad x*y=\omega x+(1-\omega)y.
\end{equation}
Since the kernel of reduction modulo $p$ is a $p$-group, the reduction of $\omega$ also has order $d$, and $1-\omega$ is a unit.
Thus $X_{p^r,\omega}$ is connected and regular.
By \eqref{low:affine-data}, its inner automorphism group and the stabilizer of $0$ are
\begin{equation}\label{cyc:groups}
 G=C_{p^r}\rtimes C_d,\qquad H=C_d
\end{equation}
and the augmentation element is $h=S_0$.
Replacing the generator of $C_d$ by its inverse interchanges $\omega$ and $\omega^{-1}$ in the semidirect-product presentation.
The pair $(G,H)$ is preserved, but the augmentation element $h$ must be tracked separately.
Put $L=\Cofib(BH\to BG)$.

In the case $r=1$, the author computed the integral quandle homology for all $d=\operatorname{ord}(\omega)>1$, including composite $d$, and gave cocycles representing generators of the cohomology with coefficients in $\F_p$ \cite[Theorem~1.1, Corollary~5.9]{Nosaka2013}.

\subsection{Cohomology and higher Bocksteins}

The following calculation supplies the input for the loop-space identification in Theorem~\ref{cyc:loop}.
We use the Bockstein spectral sequence associated with the $p$-adic filtration on integral cochains, with the ordinary mod-$p$ Bockstein denoted by $\beta_1$.
A statement about $\beta_r$ includes the assertion that the specified classes survive to the $r$th page.
We write $\Lambda(-)$ for the exterior algebra over $\F_p$; subscripts on algebra generators indicate their degrees.

\begin{lemma}\label{cyc:invariants}
Generators $z$ and $w$ may be chosen so that
\begin{equation}\label{cyc:group-cohomology}
 H^*(G;\F_p)=\Lambda(z_{2d-1})\otimes\F_p[w_{2d}],
 \qquad \beta_r(z)=w.
\end{equation}
All Bockstein differentials before $\beta_r$ vanish.
\end{lemma}

\begin{proof}
For $P=C_{p^r}$, the periodic resolution gives $H^*(P;\F_p)=\Lambda(a_1)\otimes\F_p[b_2]$.
The integral cochain differential alternates between $0$ and $p^r$, so all differentials before the $r$th page vanish, and generators may be chosen such that $\beta_r(a)=b$ and $\beta_r(b)=0$.
Since $p\nmid d$, restriction and averaging give
\[
 H^*(G;\F_p)\cong H^*(P;\F_p)^{C_d}.
\]
By naturality of $H^1(P;\F_p)=\operatorname{Hom}(P,\F_p)$ and $\beta_r$, a generator of $C_d$ acts on $a,b$ by the same scalar of order $d$.
The invariant monomials are $b^{jd}$ for $j\ge0$ and $ab^{jd-1}$ for $j\ge1$, so we may take $z=ab^{d-1}$ and $w=b^d$.
Since $p$ is odd, $z^2=0$.
Averaging makes passage to invariants exact on every page, and the derivation rule gives $\beta_r(z)=w$.
Replacing the generator of the quotient group by its inverse does not change the invariants.
\end{proof}

\subsection{The loop space of the homotopy cofiber}

For odd $m>1$ and an integer $N\ge1$, let $S^m\{N\}$ be the homotopy fiber of a based degree-$N$ map $[N]:S^m\to S^m$.
Its localization satisfies
\begin{equation}\label{cyc:sphere-fibre-definition}
 S^m\{N\}_{(p)}\simeq\hofib\bigl([N]_{(p)}:S^m_{(p)}\to S^m_{(p)}\bigr).
\end{equation}
For an odd prime $p$ and $d>1$ dividing $p-1$, the space $S^{2d-1}\{p^r\}$ is a loop space.
Moreover, if a $p$-complete space $Z$ has the mod-$p$ cohomology algebra and higher Bocksteins of Lemma~\ref{cyc:invariants}, then $\Omega Z\simeq S^{2d-1}\{p^r\}$ \cite[Corollary~1.2 and Lemma~2.1]{BrotoLevi}.
This characterization requires the classes to survive to the $r$th page, but does not require $d$ to be prime.

\begin{theorem}\label{cyc:loop}
For $X=X_{p^r,\omega}$ in \eqref{cyc:definition}, there is an ordinary homotopy equivalence
\begin{equation}\label{cyc:loop-equivalence}
 \Omega L_{(p)}\simeq S^{2d-1}\{p^r\}_{(p)}
\end{equation}
and $\cM_{(p)}$ admits the $A_\infty$ model
\begin{equation}\label{cyc:monoid-model}
 \cM_{(p)}\simeq_{A_\infty}
 \Omega\bigl((S^2\vee B(C_{p^r}\rtimes C_d)^\wedge_p)_{(p)}\bigr).
\end{equation}
Moreover,
\begin{equation}\label{cyc:loop-low-homotopy}
 \pi_j(\Omega L_{(p)})=0\quad(1\le j<2d-2),\qquad
 \pi_{2d-2}(\Omega L_{(p)})\cong C_{p^r}.
\end{equation}
\end{theorem}

\begin{proof}
By Proposition~\ref{loc:completion}, $L_{(p)}\simeq(BG)^\wedge_p$ is simply connected and $p$-complete.
By Lemma~\ref{loc:L-properties} and localization,
$BG\to L_{(p)}$ induces an isomorphism in mod-$p$
cohomology. Naturality identifies the corresponding
Bockstein spectral sequences.
Thus Lemma~\ref{cyc:invariants} and the Broto--Levi
characterization identify the loop space.
The fiber of a degree-$p^r$ map on an odd-dimensional sphere is equivalent to both its localization and its completion \cite[Section~1.1]{BrotoLevi}, giving \eqref{cyc:loop-equivalence}.
The homotopy exact sequence of \eqref{cyc:sphere-fibre-definition} gives \eqref{cyc:loop-low-homotopy}.
Proposition~\ref{loc:completion} gives \eqref{cyc:monoid-model}.
\end{proof}

\subsection{Integral input for rack and quandle homology}

Henceforth, put $R=\Z_{(p)}$ and denote the $R$-module $R/p^rR$ by $C_{p^r}$.

\begin{lemma}\label{cyc:integral-fibre}\label{cyc:fibre-series}
Let $m=2d-1$ and $F=S^m\{p^r\}_{(p)}$.
Then
\begin{equation}\label{cyc:integral-fibre-groups}
 H_j(F;R)\cong
 \begin{cases}
 R,&j=0,\\
 C_{p^r},&j=k(2d-2)\text{ for some }k\ge1,\\
 0,&\text{otherwise}.
 \end{cases}
\end{equation}
In particular, these are also the groups $H_j(\Omega L;R)$ for $X=X_{p^r,\omega}$.
With coefficients in $\F_p$,
\begin{equation}\label{cyc:A-series}
 A_X(t)=\frac{1+t^{2d-1}}{1-t^{2d-2}}.
\end{equation}
\end{lemma}

\begin{proof}
Formula~\eqref{cyc:integral-fibre-groups} is the $p$-localization of the integral homology computation for the homotopy fiber of a degree map on a sphere \cite[Section~1.1]{BrotoLevi}.
By Theorem~\ref{cyc:loop} and invariance of homology with coefficients in $R$ under localization, the same formula computes $H_j(\Omega L;R)$.
The universal coefficient theorem gives $A_X(t)=1+(1+t)t^{2d-2}/(1-t^{2d-2})$, which is \eqref{cyc:A-series}.
\end{proof}

For any field $k$ of characteristic $p$, Benson and Greenlees determined the Pontryagin algebra $H_*(\Omega BG_p^\wedge;k)$ together with its minimal $A_\infty$-algebra structure; its Poincar\'e series agrees with \eqref{cyc:A-series} \cite[Theorem~1.3 and Section~5]{BensonGreenlees2021}.

\begin{theorem}\label{cyc:integral-homology}
For $X=X_{p^r,\omega}$ with $d>1$ dividing $p-1$,
\begin{equation}\label{cyc:integral-quandle-groups}
 H_1^Q(X;\Z)\cong\Z,\qquad
 H_n^Q(X;\Z)\cong(C_{p^r})^{b_n}\quad(n\ge2),
\end{equation}
where the multiplicities are determined by
\begin{equation}\label{cyc:integral-multiplicity-series}
 \sum_{n\ge2}b_nt^n
   =\frac{t^{2d-1}(1+t)}
          {1-t^{2d-2}-t^{2d-1}-t^{2d}}.
\end{equation}
\end{theorem}

\begin{proof}
In Corollary~\ref{hom:exponent-profiles}, Lemma~\ref{cyc:integral-fibre} gives
\[
 K_{\ge 1}(t)=\cdots=K_{\ge r}(t)=\frac{t^{2d-2}}{1-t^{2d-2}},
 \qquad K_{\ge r+1}(t)=0.
\]
Thus $Q_{\ge1}=\cdots=Q_{\ge r}$ equals the right-hand side of \eqref{cyc:integral-multiplicity-series}, and $Q_{\ge r+1}=0$.
Consequently, the only torsion summands are $C_{p^r}$, and the free part is $H_1^Q(X;R)=R$.
Since $X$ is a finite connected Alexander quandle of order $p^r$, \eqref{cyc:integral-quandle-groups} follows from the \hyperref[hom:alexander-integral-recovery]{recovery of integral homology}.
\end{proof}

\begin{corollary}\label{cyc:series}
For $X=X_{p^r,\omega}$, the mod-$p$ Poincar\'e series of the covering monoid and the rack space are
\begin{align}
 P_\cM(t)
  &=\frac{1+t^{2d-1}}{1-t-t^{2d-2}-t^{2d}},\label{cyc:M-series}\\
 R_X(t)
  &=\frac{1-t^{2d-2}}{1-t-t^{2d-2}-t^{2d}}.\label{cyc:R-series}
\end{align}
The same series give the corresponding cohomology dimensions.
Let $b_n$ be the multiplicities in Theorem~\ref{cyc:integral-homology}, with $b_1=0$.
Then
\[
 \dim_{\F_p}H_n^Q(X;\F_p)=\dim_{\F_p}H_Q^n(X;\F_p)
 =\begin{cases}
 1,&n=1,\\
 b_n+b_{n-1},&n\ge2.
 \end{cases}
\]
\end{corollary}

\begin{proof}
The two series follow by substituting \eqref{cyc:A-series} into Lemma~\ref{hom:pid-words} with $\Lambda=\F_p$.
The quandle dimensions follow from Theorem~\ref{cyc:integral-homology} and the universal coefficient theorem.
\end{proof}

\begin{example}[The dihedral family and parameter lifts]\label{cyc:dihedral}\label{cyc:parameter-lifts}\label{cyc:prime-parameters}
Let $p$ be an odd prime and $d>1$ any divisor of $p-1$, and choose an element $\bar\omega\in\F_p^\times$ of order $d$.
By Hensel's lemma, for every $r\ge1$ the simple root $\bar\omega$ of $u^d-1$ lifts uniquely to a root $\omega_r\in(\Z/p^r)^\times$, necessarily of order $d$.
Put $X_r=X_{p^r,\omega_r}$.
Theorem~\ref{cyc:integral-homology} gives
\[
 H_n^Q(X_r;\Z)=0\quad(2\le n\le2d-2),\qquad
 H_{2d-1}^Q(X_r;\Z)\cong H_{2d}^Q(X_r;\Z)\cong C_{p^r}.
\]
Thus the first degree at least two in which torsion occurs is $2d-1$.
By the universal coefficient theorem, both homology and cohomology with coefficients in $\F_p$ have dimension $1$ in degree $2d-1$ and dimension $2$ in degree $2d$.

The case $d=2$, namely $\omega_r=-1$, is the dihedral quandle; the dimension formula in Corollary~\ref{cyc:series} yields
\[
 Q_{X_r,+}(t)=\frac{t+t^4}{1-t^2-t^3-t^4}.
\]
The positive-degree dimensions begin $1,0,1,2,2,3,5,7,10,\ldots$.
\end{example}

For the prime-power family $X_{p^r,\omega}$, the field-coefficient dimensions do not distinguish the homotopy types.

\begin{proposition}\label{cyc:dimension-versus-type}
For fixed $p$ and $d$, the mod-$p$ homology dimensions in Corollary~\ref{cyc:series} are independent of $r$, whereas the ordinary homotopy types of $\Omega L_{(p)}$ and $\cM_{(p)}$ are each pairwise distinct as $r$ varies.
More precisely, put $m=2d-1$.
Then
\begin{equation}\label{cyc:M-distinguishing-homotopy}
 \pi_{m-1}(\cM_{(p)})\cong
 \pi_m(S^2)_{(p)}\oplus C_{p^r}.
\end{equation}
The graded algebras $H^*(C_{p^r}\rtimes C_d;\F_p)$ are isomorphic for all $r$, but their Bockstein spectral sequences distinguish $r$.
\end{proposition}

\begin{proof}
The assertions about the dimensions and Bocksteins follow from Corollary~\ref{cyc:series} and Lemma~\ref{cyc:invariants}.
By \eqref{cyc:loop-low-homotopy}, $L_{(p)}$ is $(m-1)$-connected with $\pi_m(L_{(p)})\cong C_{p^r}$.
Choose a CW model of $L_{(p)}$ with no positive-dimensional cells below dimension $m$.
The relative cells of $(S^2\times L_{(p)},S^2\vee L_{(p)})$ then have dimension at least $m+2$, so the wedge inclusion induces an isomorphism on $\pi_m$.
After $p$-localization, Theorem~\ref{loc:main} gives \eqref{cyc:M-distinguishing-homotopy}.
Since the torsion subgroup of $\pi_m(S^2)_{(p)}$ is finite and independent of $r$, the order of the torsion subgroup in \eqref{cyc:M-distinguishing-homotopy} determines $r$.
The ordinary homotopy type of $\Omega L_{(p)}$ also determines $r$ by \eqref{cyc:loop-low-homotopy}.
\end{proof}

\subsection{Reduction to cyclic Sylow subgroups}

The same computation extends to the case of cyclic Sylow subgroups.
The inertial index of the Sylow subgroup then plays the role of $d$.

\begin{theorem}\label{cyc:sylow-homology}
Let $p$ be a prime and $X$ a finite connected quandle.
Put $G=\Inn(X)$ and let $H$ be a point stabilizer.
Assume $p\nmid|H|$, that a Sylow $p$-subgroup $P$ of $G$ is cyclic of order $p^r$, and that $r\ge1$.
Put
\begin{equation}\label{cyc:sylow-inertial-index}
 e=|N_G(P):C_G(P)|.
\end{equation}
Then $p$ is odd, $e>1$, and $e\mid p-1$.
For $L=\Cofib(BH\to BG)$, there is an ordinary homotopy equivalence
\begin{equation}\label{cyc:sylow-fibre}
 \Omega L_{(p)}\simeq S^{2e-1}\{p^r\}_{(p)}.
\end{equation}
If $\omega\in(\Z/p^r)^\times$ has order $e$, there is also an $A_\infty$ equivalence of the specified monoids
\begin{equation}\label{cyc:sylow-monoid}
 (\cM_X)_{(p)}\simeq_{A_\infty}(\cM_{X_{p^r,\omega}})_{(p)}.
\end{equation}
The $p$-local quandle homology of $X$ is
\begin{equation}\label{cyc:sylow-integral-groups}
 H_1^Q(X;\Z_{(p)})\cong\Z_{(p)},\qquad
 H_n^Q(X;\Z_{(p)})\cong(C_{p^r})^{b_n(e)}\quad(n\ge2).
\end{equation}
Here $b_n(e)$ is the multiplicity in \eqref{cyc:integral-multiplicity-series} with $d=e$.
The mod-$p$ formulas of Corollary~\ref{cyc:series} also hold with $d=e$.
\end{theorem}

\begin{proof}
The cyclic Sylow reduction in \cite[Section~1]{BensonGreenlees2021} gives
\[
 BG_p^\wedge\simeq B\Gamma_p^\wedge,
 \qquad \Gamma=C_{p^r}\rtimes C_e,
\]
where $e\mid p-1$ and $C_e$ acts faithfully on $C_{p^r}$.
By Proposition~\ref{loc:completion}, $BG_p^\wedge$ is simply connected.
If $e=1$, then $B\Gamma_p^\wedge\simeq BC_{p^r}$, since the classifying space of a finite $p$-group is already $p$-complete.
Its fundamental group is nontrivial because $r\ge1$, a contradiction.
Thus $e>1$, and hence $p$ is odd.

Choose $\omega\in(\Z/p^r)^\times$ of order $e$.
Then $\Gamma\cong\Inn(X_{p^r,\omega})$, with point stabilizer $C_e$.
Proposition~\ref{loc:completion} and Theorem~\ref{cyc:loop} give \eqref{cyc:sylow-fibre} and \eqref{cyc:sylow-monoid}.
The monoid equivalence and Proposition~\ref{hom:shift} identify the rack homology groups over $\Z_{(p)}$ with those of $X_{p^r,\omega}$.
Corollary~\ref{hom:rack-determines-quandle} therefore gives isomorphic quandle homology groups.
Localizing Theorem~\ref{cyc:integral-homology} for $X_{p^r,\omega}$ yields \eqref{cyc:sylow-integral-groups}.
Over $\F_p$, the same monoid equivalence and Proposition~\ref{hom:shift} give the $P_\cM$ and $R_X$ formulas of Corollary~\ref{cyc:series}; the universal coefficient theorem applied to \eqref{cyc:sylow-integral-groups} gives its quandle dimensions.
\end{proof}

\begin{corollary}[Regular cyclic Alexander quandles]
\label{cyc:regular-cyclic-integral}
Let $N>1$, let $\omega\in(\Z/N)^\times$, and put
$D=\operatorname{ord}_N(\omega)$.
Assume
\[
 \gcd(N,1-\omega)=\gcd(N,D)=1,
\]
and equip $X=\Z/N$ with $x*y=\omega x+(1-\omega)y$.
Write $N=\prod_{p\mid N}p^{r_p}$ and put
$d_p=\operatorname{ord}_{p^{r_p}}(\omega)$.
Then every $p\mid N$ is odd, $d_p>1$ divides $p-1$, and
\begin{equation}\label{cyc:regular-cyclic-groups}
 H_1^Q(X;\Z)\cong\Z,\qquad
 H_n^Q(X;\Z)\cong
 \bigoplus_{p\mid N}(C_{p^{r_p}})^{b_n(d_p)}
 \quad(n\ge2),
\end{equation}
where $b_n(d)$ is determined by \eqref{cyc:integral-multiplicity-series}.
In particular, for the dihedral quandle $R_N$ of any odd order $N>1$,
\begin{equation}\label{cyc:odd-dihedral-integral}
 H_n^Q(R_N;\Z)\cong(C_N)^{b_n(2)}\quad(n\ge2).
\end{equation}
\end{corollary}

\begin{proof}
The first coprimality assumption makes $X$ connected, and
\eqref{low:affine-data} gives $G=C_N\rtimes C_D$ and $H=C_D$.
For each $p\mid N$, the assumption $p\nmid D$ makes the translation subgroup $P=C_{p^{r_p}}$ a normal Sylow $p$-subgroup and gives $p\nmid|H|$.
The image of $C_D$ in $\operatorname{Aut}(P)$ has order $d_p$, so the inertial index is $d_p$.
Theorem~\ref{cyc:sylow-homology} therefore gives the asserted
$p$-primary groups and the conditions on $p$ and $d_p$.
The \hyperref[hom:alexander-integral-recovery]{recovery of integral homology}
combines these primary components into \eqref{cyc:regular-cyclic-groups}.
For $\omega=-1$ and odd $N$, all $d_p$ equal $2$;
the Chinese remainder theorem gives \eqref{cyc:odd-dihedral-integral}.
\end{proof}

\begin{example}\label{cyc:mixed-local-orders}
For $X=\Z/15$ with $x*y=2x-y$, the global order of $\omega=2$
is $4$, while $d_3=2$ and $d_5=4$.
Hence
\[
 H_n^Q(X;\Z)\cong(C_3)^{b_n(2)}\oplus(C_5)^{b_n(4)}
 \quad(n\ge2).
\]
The first $3$-primary torsion occurs in degree $3$, and the first
$5$-primary torsion occurs in degree $7$.
\end{example}

\section{The tetrahedral quandle}\label{tet:section}

We prove Theorem~\ref{intro:tetrahedral} by identifying a homotopy-fiber model and computing its homology over $\Z_{(2)}$.

Let $\F_4=\F_2[\omega]/(\omega^2+\omega+1)$ and equip its underlying set with the Alexander quandle operation
\[
 x*y=\omega x+(1-\omega)y.
\]
Choosing the other primitive element gives an isomorphic quandle via the Frobenius automorphism.
Denote this quandle by $X$ and put
\[
 G=\Inn(X),\qquad H=\operatorname{Stab}_G(0),\qquad
 \cM=B(G;X),\qquad L=\Cofib(BH\longrightarrow BG).
\]
By \eqref{low:affine-data}, $G$ is generated by translations and multiplication by $\omega$.
These are even permutations of four letters, and the translation group is the Klein four-group; thus
\begin{equation}\label{tet:groups}
 G=\F_4\rtimes\langle\omega\rangle\cong A_4,
 \qquad H=\langle\omega\rangle\cong C_3.
\end{equation}
Lemma~\ref{loc:L-properties} shows that $L$ is simply connected.

Let $c:\mathbb RP^3\to S^3$ be the based map of degree $1$ obtained by collapsing the standard $2$-skeleton, with compatible orientations.
Define
\begin{equation}\label{tet:definition-K}
 K_{\mathrm{tet}}=\hofib\!\bigl(
 \mathbb{R}P^3_{(2)}
 \xrightarrow{\, [4]_{(2)}\circ c_{(2)}\,}
 S^3_{(2)}\bigr).
\end{equation}
Here the homotopy fiber is based, and the degree $4$ refers to the action on $H_3$; the action on $\pi_3$ is multiplication by $8$, as shown below.

The fiber sequence $(S^3)^\wedge_2\to(BA_4)^\wedge_2\to BSO(3)^\wedge_2$ and the homology of $\Omega((BA_4)^\wedge_2)$ with coefficients in $\F_2$ are known \cite[Example~13.3]{BensonGreenleesShamir2013}.
Below, we identify the degree-$4$ map geometrically and compare it with the quandle model.

\begin{theorem}\label{tet:loop-model}
There is a homotopy equivalence of underlying spaces
\begin{equation}\label{tet:ordinary-equivalence}
 \Omega L_{(2)}\simeq K_{\mathrm{tet}}.
\end{equation}
The space $K_{\mathrm{tet}}$ is nilpotent and $2$-complete, and
\begin{equation}\label{tet:low-homotopy}
 \pi_1K_{\mathrm{tet}}\cong C_2,\qquad \pi_2K_{\mathrm{tet}}\cong C_8.
\end{equation}
In the notation of \eqref{cyc:sphere-fibre-definition}, its universal cover has the ordinary homotopy type
\begin{equation}\label{tet:universal-cover}
 \widetilde K_{\mathrm{tet}}\simeq S^3\{8\}_{(2)}.
\end{equation}
The loop space comparison for the covering monoid is the $A_\infty$ equivalence
\begin{equation}\label{tet:monoid-loop}
 \cM_{(2)}\simeq_{A_\infty}\Omega\bigl((S^2\vee L)_{(2)}\bigr).
\end{equation}
An $A_\infty$ structure on $K_{\mathrm{tet}}$ is specified only when the multiplication is transferred through \eqref{tet:ordinary-equivalence}.
\end{theorem}

\begin{proof}
Realize $A_4$ as the rotation group of an oriented regular tetrahedron and put
\[
 Q=SO(3)/A_4\cong S^3/(2T),\qquad |2T|=24.
\]
For the binary tetrahedral group $2T\cong Q_8\rtimes C_3$, the $C_3$-coinvariant quotient of $Q_8/[Q_8,Q_8]\cong C_2^2$ vanishes, so $(2T)_{\mathrm{ab}}\cong C_3$.
By Poincar\'e duality, $Q$ is a homology $3$-sphere with coefficients in $\F_2$.
Thus the orientation-preserving degree-$1$ collapse map $e:Q\to S^3$ is an $\F_2$-homology equivalence, and $e^\wedge_2:Q^\wedge_2\to(S^3)^\wedge_2$ is an equivalence.

The inclusion $A_4\subset SO(3)$ gives the homogeneous-space fiber sequence
\begin{equation}\label{tet:homogeneous-fibration}
 Q\longrightarrow BA_4\longrightarrow BSO(3).
\end{equation}
Since $BSO(3)$ is simply connected, the mod-$2$ fiber lemma \cite[Chapter II, \S5]{BK} applies both to this sequence and to the path-space fibration of $BSO(3)$.
Thus completion preserves the sequence and gives $\Omega((BSO(3))^\wedge_2)\simeq SO(3)^\wedge_2$.
Under $SO(3)\simeq\Omega BSO(3)$, the connecting map is the quotient $q:SO(3)\to Q$.
Using $e^\wedge_2$, we identify
$\Omega((BA_4)^\wedge_2)$ with the homotopy fiber of
$(e\circ q)^\wedge_2$.
Under $SO(3)\cong\mathbb RP^3$, the covering $q$ has
degree $12$.
Obstruction theory gives
$[\mathbb RP^3,S^3]\cong H^3(\mathbb RP^3;\Z)\cong\Z$,
so $e\circ q\simeq[12]\circ c\simeq[3]\circ[4]\circ c$.
The completed degree-$3$ map of $S^3$ is a homotopy
equivalence. Postcomposing with its homotopy inverse
therefore gives
\begin{equation}\label{tet:completed-degree-four}
 \Omega((BA_4)^\wedge_2)\simeq
 \hofib\!\bigl((\mathbb RP^3)^\wedge_2
       \xrightarrow{\,([4]\circ c)^\wedge_2\,}
       (S^3)^\wedge_2\bigr).
\end{equation}

For the localized fiber, let $\rho:S^3\to\mathbb RP^3$ be the universal covering and put $f=[4]_{(2)}\circ c_{(2)}$.
Since $\mathbb RP^3\cong SO(3)$ is simple, $\rho_{(2)}$ models the universal cover of $\mathbb RP^3_{(2)}$.
The composite $f\circ\rho_{(2)}$ has degree $8$, so the homotopy exact sequence gives $\pi_1K_{\mathrm{tet}}\cong C_2$ and $\pi_2K_{\mathrm{tet}}\cong\Z_{(2)}/8\Z_{(2)}\cong C_8$.
The map $K_{\mathrm{tet}}\to\mathbb RP^3_{(2)}$ induces an isomorphism on fundamental groups, so the pullback along $\rho_{(2)}$ is $\widetilde K_{\mathrm{tet}}$.
This is the fiber of the composite map of degree $8$, giving \eqref{tet:universal-cover}.

Since $f$ has nilpotent domain $\mathbb RP^3_{(2)}$, its homotopy fiber $K_{\mathrm{tet}}$ is nilpotent \cite[Proposition~4.4.1(i)]{MayPonto}.
The universal-cover description, injectivity of multiplication by $8$ on $\Z_{(2)}$, and finiteness of $\pi_n(S^3)$ for $n>3$ show that all homotopy groups of $K_{\mathrm{tet}}$ are finite $2$-groups.
Hence $K_{\mathrm{tet}}$ is $2$-complete \cite[Theorem~2.6(1)]{BarthelBousfield}.
Since $S^3_{(2)}$ is simply connected, the mod-$2$ fiber lemma identifies $(K_{\mathrm{tet}})^\wedge_2$ with the homotopy fiber of $f^\wedge_2$.
The $2$-completions of $\mathbb RP^3$ and $S^3$ are unchanged by $2$-localization, so the $2$-completeness of $K_{\mathrm{tet}}$ identifies $K_{\mathrm{tet}}$ with the fiber in \eqref{tet:completed-degree-four}.
Finally, $|H|=3$ and Proposition~\ref{loc:completion} give $L_{(2)}\simeq(BA_4)^\wedge_2$, proving \eqref{tet:ordinary-equivalence}, and Theorem~\ref{loc:main} gives \eqref{tet:monoid-loop}.
\end{proof}

\begin{remark}\label{tet:broto-levi}
The universal-cover computation agrees with \cite[Proposition~2.2]{BrotoLevi}.
This geometric comparison also retains the map on $\mathbb RP^3$ needed to recover $\Omega L_{(2)}$ itself.
For the rack space of the tetrahedral quandle, the author had already obtained $\pi_2(\mathcal{R}_X)\cong\Z\oplus C_2\oplus C_8$ \cite[Proposition~7.10, equation~(18)]{Nosaka2015Homotopical}.
\end{remark}

\subsection{Fiber homology and integral quandle homology}\label{tet:additive}

We compute the homology of the fiber with coefficients in $\Z_{(2)}$ and determine the integral quandle homology of the tetrahedral quandle in all positive degrees.

\begin{lemma}\label{tet:integral-K}
Put $R=\Z_{(2)}$.
The homology of the fiber $K_{\mathrm{tet}}$ in \eqref{tet:definition-K} is
\begin{equation}\label{tet:integral-K-groups}
 H_0(K_{\mathrm{tet}};R)=R,\qquad
 H_{2j}(K_{\mathrm{tet}};R)\cong C_4\quad(j\geq1),\qquad
 H_{2j+1}(K_{\mathrm{tet}};R)\cong C_2\quad(j\geq0).
\end{equation}
\end{lemma}

\begin{proof}
The defining map $f=[4]_{(2)}\circ c_{(2)}$ gives the pullback fibration
\begin{equation}\label{tet:Serre-fibration}
 \Omega S^3_{(2)}\longrightarrow K_{\mathrm{tet}}
 \longrightarrow\mathbb{R}P^3_{(2)}.
\end{equation}
The local coefficient action factors through $\pi_1(S^3_{(2)})=0$, so the Serre spectral sequence is
\[
 E^2_{a,b}=H_a(\mathbb{R}P^3_{(2)};R)
                  \otimes_R H_b(\Omega S^3_{(2)};R).
\]
The nonzero columns are $a=0,1,3$, with groups $R,C_2,R$, respectively, in every nonnegative even row.
Thus $d_2=0$, and the only possibly nonzero subsequent differentials are
\[
 d_3:E^3_{3,2j}\longrightarrow E^3_{0,2j+2}
       \qquad(j\geq0).
\]
In the path-space fibration over $S^3_{(2)}$, contractibility of the total space makes the corresponding $d_3$ differentials isomorphisms.
The comparison is the identity on the fiber and multiplication by $4$ on $H_3$ of the base, so naturality makes each displayed differential multiplication by $4$ up to a unit.
Consequently, in positive degrees the surviving terms are $C_4$ in column $0$ and $C_2$ in column $1$.
There is only one term in each total degree, so there are no additive extension problems.
\end{proof}

\begin{proof}[Proof of Theorem~\ref{intro:tetrahedral}]
Put $R=\Z_{(2)}$.
The ordinary equivalence in Theorem~\ref{tet:loop-model} and Lemma~\ref{tet:integral-K} give the following input to Corollary~\ref{hom:exponent-profiles}:
\begin{equation}\label{tet:integral-kappa}
 K_{\ge1}(t)=\frac{t}{1-t},\qquad
 K_{\ge2}(t)=\frac{t^2}{1-t^2},\qquad K_{\ge e}(t)=0\quad(e\ge3).
\end{equation}
Since $|H|=3$ is a unit in $R$, substitution in \eqref{hom:exponent-profile-series} gives
\[
 Q_{\ge1}(t)=\frac{t^2(1+t)}{1-t-t^2-t^3},\qquad
 Q_{\ge2}(t)=\frac{t^3}{1-t-t^3},\qquad
 Q_{\ge e}(t)=0\quad(e\ge3).
\]
Thus the only torsion summands are $C_2$ and $C_4$.
Their multiplicities $a_n$ and $b_n$ in degree $n\ge2$ are the coefficients of $t^n$ in $Q_{\ge1}-Q_{\ge2}$ and $Q_{\ge2}$, respectively.
Since $X$ is a connected Alexander quandle of order $4$, the \hyperref[hom:alexander-integral-recovery]{recovery of integral homology} gives the asserted integral groups.
\end{proof}

The universal coefficient theorem also gives $H_Q^1(X;\Z)\cong\Z$, $H_Q^2(X;\Z)=0$, and $H_Q^n(X;\Z)\cong C_2^{a_{n-1}}\oplus C_4^{b_{n-1}}$ for $n\ge3$.

\begin{remark}[Low-degree computations and an existing conjecture]\label{tet:integral-check}
Put $a_1=b_1=0$, $s_n=a_n+b_n$, and $q_n=\dim_{\F_2}H_Q^n(X;\F_2)$.
Theorem~\ref{intro:tetrahedral} and the universal coefficient theorem give the following low-degree values.
\begin{center}
\begin{tabular}{r*{8}{r}}
\toprule
$n$&1&2&3&4&5&6&7&8\\
\midrule
$a_n$&0&1&1&2&5&9&17&33\\
$b_n$&0&0&1&1&1&2&3&4\\
$s_n$&0&1&2&3&6&11&20&37\\
$q_n$&1&1&3&5&9&17&31&57\\
\bottomrule
\end{tabular}
\end{center}

The groups through degree $7$ agree with the computations of Niebrzydowski--Przytycki \cite[Example~13(ii)]{NiebrzydowskiPrzytycki2010}, who denote the tetrahedral quandle by $S_4$.

Theorem~\ref{intro:tetrahedral} proves their conjectured recurrence $b_n=b_{n-1}+b_{n-3}$ $(n\ge4)$ for the multiplicities of the $C_4$ summands.
In degree $8$, however, the group is $H_8^Q(X;\Z)\cong C_2^{33}\oplus C_4^4$, differing from their prediction $C_2^{32}\oplus C_4^4$ by one $C_2$ direct summand.
Their prediction for the entire group therefore requires a correction.
The corresponding mod-$2$ dimension is $57$ in the table, whereas their prediction would give $56$.
\end{remark}

\section{A product Alexander quandle}
\label{ex:comparisons}

The following example illustrates the tensor and Tor terms in the integral
calculation when the Sylow subgroup is not cyclic.

\begin{example}\label{ex:product-cyclic-blocks}
Let $X=(\F_7^2,T)$ with $T=\operatorname{diag}(-1,2)$ and
$x*y=Tx+(1-T)y$.
Then $X$ is a connected regular Alexander quandle of order $49$, and
\[
 H_1^Q(X;\Z)\cong\Z,\qquad
 H_n^Q(X;\Z)\cong C_7^{\,b_n}\quad(n\ge2),
\]
where, on putting
\[
 U(t)=\frac{t^2}{1-t^2}+\frac{t^4}{1-t^4}
       +\frac{(1+t)t^6}{(1-t^2)(1-t^4)},
\]
the multiplicities are determined by
\[
 \sum_{n\ge2}b_nt^n
 =\frac{t(1+t)U(t)}{1-t(1+t)U(t)}
 =t^3+t^4+2t^5+3t^6+4t^7+8t^8+O(t^9).
\]
In particular, $H_n^Q(X;\Z)$ is annihilated by $7$ for every $n\ge2$.
\end{example}

\begin{proof}
Write $X=X_1\times X_2$ with $X_1=(\F_7,-1)$ and $X_2=(\F_7,2)$.
The two automorphisms have relatively prime orders $2$ and $3$, so
$\langle T\rangle=\langle-1\rangle\times\langle2\rangle$.
Equation~\eqref{low:affine-data} therefore gives
\[
 G=\Inn(X)\cong(C_7\rtimes C_2)\times(C_7\rtimes C_3),
 \qquad H=\operatorname{Stab}_G(0)\cong C_6.
\]
Put $G_i=\Inn(X_i)$, $H_i=\operatorname{Stab}_{G_i}(0)$, and
$L_i=\Cofib(BH_i\to BG_i)$; also put $L=\Cofib(BH\to BG)$.
By Proposition~\ref{loc:completion}, the product map
$BG_1\times BG_2\to(L_1)_{(7)}\times(L_2)_{(7)}$
is a mod-$7$ homology equivalence with simply connected, $7$-complete target.
Completing this map and applying the same proposition to $(G,H)$ gives
$L_{(7)}\simeq(L_1)_{(7)}\times(L_2)_{(7)}$.
Thus Theorem~\ref{cyc:loop} yields the ordinary homotopy equivalence
\[
 K:=\Omega L_{(7)}\simeq S^3\{7\}_{(7)}\times S^5\{7\}_{(7)}.
\]

Put $R=\Z_{(7)}$.
Lemma~\ref{cyc:integral-fibre} and the split K\"unneth theorem show
that all positive-degree homology groups of $K$ are sums of $C_7$.
Their multiplicity series is $U(t)$: the first two terms come from
the factors, and the last term records both their tensor product and
$\operatorname{Tor}_1^R(C_7,C_7)\cong C_7$, one degree higher.
Hence $K_{\ge1}=U$ and $K_{\ge e}=0$ for $e\ge2$ in
Corollary~\ref{hom:exponent-profiles}.
That corollary and the
\hyperref[hom:alexander-integral-recovery]{recovery of integral homology}
give the asserted integral groups.
\end{proof}

For any field $k$ of characteristic $7$, the universal coefficient theorem gives
\[
 A_X(t)=\sum_{n\ge0}\dim_k H_n(\Omega L;k)t^n
 =1+(1+t)U(t)
 =\frac{(1+t^3)(1+t^5)}{(1-t^2)(1-t^4)}.
\]
Substituting $\kappa=A_X$ into \eqref{hom:pid-series} with $\Lambda=k$ determines all rack homology dimensions and all positive-degree quandle homology dimensions over $k$.

\section{Two non-Alexander families}\label{ex:nonalex-section}
\label{ex:non-alexander-homology}

The same local calculation applies to two families of non-Alexander quandles.
For an integer $e\ge2$, retain $b_n(e)$ from Theorem~\ref{cyc:sylow-homology} and define $a_n(e)$ by the first series below; we recall the second series for comparison:
\begin{equation}\label{ex:nonalex-multiplicities}
 \begin{aligned}
 \sum_{n\ge1}a_n(e)t^n
   &=\frac{t^{2e-1}}
     {(1-t)(1-t-t^{2e-2}-t^{2e})},\\
 \sum_{n\ge2}b_n(e)t^n
   &=\frac{t^{2e-1}(1+t)}
     {1-t^{2e-2}-t^{2e-1}-t^{2e}}.
 \end{aligned}
\end{equation}
Let $X$ be a finite connected quandle, put $G=\Inn(X)$ and
$H=\operatorname{Stab}_G(x_0)$, and set $L=\Cofib(BH\to BG)$.
Let $\ell$ be a prime with $\ell\nmid|H|$, and suppose that there is an ordinary homotopy equivalence
\[
 \Omega L_{(\ell)}\simeq S^{2e-1}\{\ell^r\}_{(\ell)},\qquad r\ge1.
\]
The positive-degree integral homology of this fiber is $C_{\ell^r}$ in degrees
divisible by $2e-2$ and zero otherwise \cite[Section~1.1]{BrotoLevi}.
Substitution into Lemma~\ref{hom:pid-words} and
Corollary~\ref{hom:exponent-profiles} therefore gives
\begin{equation}\label{ex:nonalex-local-groups}
 \begin{aligned}
 H_n^R(X;\Z_{(\ell)})
   &\cong\Z_{(\ell)}\oplus(C_{\ell^r})^{a_n(e)}&&(n\ge1),\\
 H_1^Q(X;\Z_{(\ell)})&\cong\Z_{(\ell)},&\\[-2pt]
 H_n^Q(X;\Z_{(\ell)})
   &\cong(C_{\ell^r})^{b_n(e)}&&(n\ge2).
 \end{aligned}
\end{equation}

To verify that the following examples are not Alexander, we use a common criterion.
In a finite connected Alexander quandle, $1-T$ is invertible, so every left translation $y\mapsto x*y$ is bijective.

\begin{example}[The conjugation quandle of transpositions]
\label{ex:transposition-homology}
Let $\ell\ge5$ be prime and $m\in\{\ell,\ell+1\}$.
Let $T_m$ be the set of transpositions in $S_m$, with $x*y=y^{-1}xy$.
Then $T_m$ is a connected non-Alexander quandle, and
\begin{equation}\label{ex:transposition-groups}
 \begin{aligned}
 H_n^R(T_m;\Z_{(\ell)})
   &\cong\Z_{(\ell)}\oplus(C_\ell)^{a_n(\ell-1)}&&(n\ge1),\\
 H_1^Q(T_m;\Z_{(\ell)})&\cong\Z_{(\ell)},&\\[-2pt]
 H_n^Q(T_m;\Z_{(\ell)})
   &\cong(C_\ell)^{b_n(\ell-1)}&&(n\ge2).
 \end{aligned}
\end{equation}
In particular,
\[
 \begin{gathered}
 H_n^Q(T_m;\Z_{(\ell)})=0\quad(2\le n\le2\ell-4),\\
 H_{2\ell-3}^Q(T_m;\Z_{(\ell)})
  \cong H_{2\ell-2}^Q(T_m;\Z_{(\ell)})\cong C_\ell.
 \end{gathered}
\]
The torsion subgroups of rack homology in these last two degrees are
$C_\ell$ and $(C_\ell)^2$, respectively.
\end{example}

\begin{proof}
Transpositions generate $S_m$, and its conjugation action on $T_m$ is
faithful and transitive. Thus
\[
 G=\Inn(T_m)\cong S_m,\qquad
 H=C_{S_m}((12))\cong C_2\times S_{m-2}.
\]
Since $m=\ell$ or $\ell+1$, we have $\ell\nmid|H|$ and a Sylow
$\ell$-subgroup $P\cong C_\ell$.
Identify its support with $\F_\ell$ so that $P$ acts by translations.
The permutations $x\mapsto ax$, $a\in\F_\ell^\times$, extended by the
identity off the support, realize every automorphism of $P$.
Hence $|N_G(P):C_G(P)|=\ell-1$.
Theorem~\ref{cyc:sylow-homology} and \eqref{ex:nonalex-local-groups},
with $e=\ell-1$ and $r=1$, give \eqref{ex:transposition-groups}.
The first nonzero groups follow by expanding \eqref{ex:nonalex-multiplicities}.
Finally, $(12)*(12)=(12)*(34)=(12)$, so the left translation by $(12)$ is not injective.
The preceding criterion shows that $T_m$ is not Alexander.
\end{proof}

\begin{example}[Rank-one symplectic quandles]
\label{ex:symplectic-rank-one-homology}
Let $q=p^f>27$, where $p$ is prime and $f\ge1$.
The symplectic quandle
\[
 X_q=\F_q^2\setminus\{0\},\qquad v*w=v+\det(v,w)w,
\]
is connected and is not Alexander.
Put $R=\Z[1/p]$ and $N=q^2-1$. Then
\begin{equation}\label{ex:symplectic-rank-one-groups}
 \begin{aligned}
 H_n^R(X_q;R)&\cong R\oplus(C_N)^{a_n(2)}&&(n\ge1),\\
 H_1^Q(X_q;R)&\cong R,&\\[-2pt]
 H_n^Q(X_q;R)&\cong(C_N)^{b_n(2)}&&(n\ge2).
 \end{aligned}
\end{equation}
Here
\[
 \sum_{n\ge1}a_n(2)t^n
  =\frac{t^3}{(1-t)(1-t-t^2-t^4)},\qquad
 \sum_{n\ge2}b_n(2)t^n=\frac{t^3}{1-t-t^3}.
\]
For example,
\[
 \begin{gathered}
 H_2^Q(X_q;R)=0,\qquad
 H_3^Q(X_q;R)\cong H_4^Q(X_q;R)\cong H_5^Q(X_q;R)\cong C_N,\\
 H_6^Q(X_q;R)\cong(C_N)^2.
 \end{gathered}
\]
Thus \eqref{ex:symplectic-rank-one-groups} determines all primary
components away from $p$ in every positive degree.
\end{example}

\begin{proof}
The linear extension of a right translation is a transvection.
Replacing its direction vector $w$ by $cw$ multiplies its parameter by $c^2$.
In characteristic $2$, every element is a square; in odd characteristic,
every element is a difference of two squares.
Composing these transvections and their inverses therefore gives all
transvections, which generate $\operatorname{SL}_2(\F_q)$.
Its natural action on nonzero vectors is faithful and transitive, and
\[
 G=\Inn(X_q)\cong\operatorname{SL}_2(\F_q),\qquad
 H=\operatorname{Stab}_G((1,0))\cong(\F_q,+)\cong(C_p)^f.
\]
In particular, $X_q$ is connected and every prime $\ell\ne p$ is good.

For each such prime, we claim that there is an ordinary homotopy equivalence
\begin{equation}\label{ex:symplectic-local-fibre}
 \Omega L_{(\ell)}\simeq\bigl(S^3\{N\}\bigr)_{(\ell)}.
\end{equation}
If $\ell$ is odd and divides $N$, a Sylow $\ell$-subgroup $P$ lies in a split or nonsplit maximal torus of $G$.
These tori are cyclic of orders $q-1$ and $q+1$, respectively, so $P$ has order $\ell^{v_\ell(N)}$, where $v_\ell(N)$ is the exponent of $\ell$ in $N$.
A nontrivial element of $P$ is noncentral and has that torus as its centralizer.
The quotient of the torus normalizer by the torus is $C_2$, acting by inversion.
Since $P$ is characteristic in the torus and $N_G(P)$ normalizes $C_G(P)$, it follows that $N_G(P)/C_G(P)\cong C_2$.
Theorem~\ref{cyc:sylow-homology} gives \eqref{ex:symplectic-local-fibre}.
If $q$ is odd and $\ell=2$, put $r=v_2(N)\ge3$.
The Broto--Levi model identifies
$\Omega\bigl((B\operatorname{SL}_2(\F_q))^\wedge_2\bigr)$ with
$S^3\{2^r\}$ \cite[Section~1.1 and Lemma~2.1]{BrotoLevi};
Proposition~\ref{loc:completion} then gives
\eqref{ex:symplectic-local-fibre} at $2$.
If $\ell\nmid pN$, both sides of \eqref{ex:symplectic-local-fibre}
are contractible, since $|G|=qN$.

For $\ell\mid N$, apply \eqref{ex:nonalex-local-groups} with
$e=2$ and $r=v_\ell(N)$.
At the remaining primes $\ell\ne p$, Lemma~\ref{hom:pid-words}
gives the free rack summand and vanishing quandle homology in degrees at least two.
The homology groups are finitely generated in each degree, so assembling
the primary components proves \eqref{ex:symplectic-rank-one-groups}.
Finally, if $c\in\F_q^\times\setminus\{1\}$, then $v*v=v*(cv)=v$.
Thus a left translation is not injective, and the same criterion shows that $X_q$ is not Alexander.
\end{proof}

\begin{remark}\label{ex:symplectic-known-low-degrees}
For odd $q>27$, the author computed
\[
 H_2^Q(X_q;\Z)\cong(C_p)^f,\qquad
 H_3^Q(X_q;\Z)\cong C_{q^2-1}\oplus(C_p)^{f(f+1)/2}
\]
\cite[proof of Theorem~7.4 and Theorem~8.6]{Nosaka2015Homotopical}.
The assumption $q>27$ avoids the exceptional values
$3,5,7,9,27$ and, for odd $q$, implies $f(p-1)>6$.
The preceding example recovers the components away from $p$ in these
low degrees and determines them in all positive degrees.
Its proof uses the rank-one local models directly and does not
require the degree to lie in a homological stability range.
\end{remark}

\section{Multiplicative structures and unstable homotopy}
\label{ainf:section}

Using the local model of Section~\ref{loc:section} and the chain model of Section~\ref{hom:section}, we construct direct summands by mixed Samelson products and determine whether the specified multiplication admits an $E_2$ refinement.
Unless otherwise stated, $X$ is a finite connected quandle in this section, and we put $G=\Inn(X)$, $H=\operatorname{Stab}_G(x_0)$, $L=\Cofib(BH\to BG)$, and $\cM=\cM_X$.
For a good prime $p$ dividing $|X|$, we also give an exponential lower bound for the number of cyclic direct summands of a fixed $p$-power order in each sufficiently large homotopy degree.

\begin{definition}\label{ainf:e2-refinement-convention}
We say that a specified grouplike $A_\infty$-space $Z$ \emph{admits a refinement to an $E_2$ structure} if $Z$ is $A_\infty$ equivalent to the underlying $A_\infty$-space of a grouplike $E_2$-space.
Thus refinement is understood up to $A_\infty$ equivalence, without requiring a strict extension on a fixed model.
Throughout this section, we use the specified product on $\cM=B(G;X)$, its lift to $\widetilde\cM$, and their localizations; other multiplications on these spaces are not considered.
\end{definition}

We compute the homotopy nilpotency of a grouplike $A_\infty$-space in a topological group model $Z$, using right-nested commutators with $[x,y]=xyx^{-1}y^{-1}$.
Infinite homotopy nilpotency means that, for every $r\ge2$, the $r$-fold right-nested commutator map $Z^r\to Z$ is not null-homotopic.

The quotient from a finite product of positive-dimensional spheres
to their smash product admits a homotopy right inverse after suspension.
Since $Z\simeq\Omega BZ$, precomposition with this quotient is injective
on based homotopy classes into $Z$.
Thus a nonzero iterated Samelson product detects a non-null iterated
commutator map.

\subsection{Spherical classes and iterated Samelson products}

For a coefficient ring $R$, a based space $W$, and $q\ge1$, let $h_R:\pi_q(W)\to H_q(W;R)$ denote the Hurewicz homomorphism with coefficients in $R$.
For homogeneous elements in a graded algebra, write $[a,b]=ab-(-1)^{|a||b|}ba$.
For $i,j\ge1$, the Samelson product of based continuous maps $S^i\to \cM$ and $S^j\to \cM$ is the commutator class on $S^i\wedge S^j$.
We use the same bracket notation for Samelson products and write $\operatorname{ad}_a(b)=[a,b]$ in both settings.
In a grouplike $A_\infty$-space, the Hurewicz homomorphism sends Samelson products to graded commutators of the Hurewicz images.

\begin{lemma}\label{ainf:higher-finiteness}
The group $\pi_j(\cM)$ is finitely generated for every $j\ge1$ and finite for $j\ge3$.
Consequently, $\pi_j(\cM_{(p)})$ is a finite $p$-group for $j\ge3$.
\end{lemma}

\begin{proof}
By Theorem~\ref{int:main}, $Y$ is simply connected and has a CW model of finite type.
Serre's theorem \cite[Theorem~4.5.4]{MayPonto} therefore shows that $\pi_j(\cM)\cong\pi_{j+1}(Y)$ is finitely generated for every $j\ge1$.
By Remark~\ref{rat:model}, $\pi_j(\cM)\otimes\Q=0$ for $j\ge3$, so these groups are finite, and localization gives their $p$-primary parts.
\end{proof}

\begin{proposition}[Integral splittings by mixed Samelson products]
\label{ainf:iterated-samelson}\label{ainf:infinite-nilpotency}
Assume $p\nmid |H|$ and $L_{(p)}\not\simeq *$, and put
\[
 K=\Omega L_{(p)},\qquad
 m=\min\{j\ge1:\pi_j(K)\ne0\},\qquad A=\pi_m(K).
\]
Then $A$ is a nonzero finite abelian $p$-group.
In the model of Theorem~\ref{loc:main}, let
$\alpha\in\pi_1(\cM_{(p)})$ be the image of
$1\in\pi_1(\Omega S^2_{(p)})\cong\Z_{(p)}$
under the looped sphere inclusion, and let
$i:K\to\cM_{(p)}$ be the loop of the other wedge inclusion.
For every $n\ge0$, the group homomorphism
\[
 \psi_n:A\longrightarrow\pi_{m+n}(\cM_{(p)}),\qquad
 \psi_n(\beta)=\operatorname{ad}_\alpha^n(i_*\beta)
\]
is a split injection.
Consequently, these iterated Samelson products realize $A$ as a direct summand of $\pi_j(\cM_{(p)})$ for every $j\ge m$ and preserve the order of each element.
Moreover, the specified product on $\cM_{(p)}$ has infinite homotopy nilpotency, and $\pi_j(\mathcal{R}_X)\otimes\Z_{(p)}\ne0$ for $j\ge\max\{m,2\}$.
\end{proposition}

\begin{proof}
Put $R=\Z_{(p)}$.
Lemma~\ref{loc:L-properties} and the Hurewicz isomorphism in the first nonzero degree show that $A$ is a nonzero finite abelian $p$-group and that $A\cong H_m(K;R)$.
This includes $m=1$, since the fundamental group of a connected loop space is abelian.
Use the inclusion-compatible chain model $T_R(u)*D$ from \eqref{hom:cobar-wedge}, where $|u|=1$, $du=0$, and $D$ is an augmented differential graded $R$-algebra modelling $K$, with $D_0=R$.
Write $D^+=\ker(D\to R)$ and also use $u$ for its homology class $h_R(\alpha)$.
Since $D=R\oplus D^+$ as chain complexes, the word decomposition \eqref{hom:pid-word-complex} shows that the words $u^nb$ with $b\in D^+$ form a direct summand of chain complexes, with differential $d(u^nb)=(-1)^nu^nd_Db$.
Projection onto these words and removal of $u^n$ induce
\begin{equation}\label{ainf:homology-projection}
 \operatorname{pr}_n:H_{m+n}(\cM_{(p)};R)\longrightarrow H_m(K;R)\cong A.
\end{equation}
For $b_\beta=h_R(\beta)$, the $u^nD^+$ component of $h_R\psi_n(\beta)=\operatorname{ad}_u^n(b_\beta)$ is $u^nb_\beta$ with coefficient one, while every other term ends in a positive power of $u$.
Hence
\[
 (\operatorname{pr}_n\circ h_R)\psi_n=\id_A.
\]
Samelson products are additive in these positive degrees, including $m=1$, so this left inverse proves the splitting and preservation of orders.
Arbitrarily long nonzero Samelson products imply infinite homotopy nilpotency.
The final assertion follows because the covering $\cM\to \mathcal{R}_X$ preserves homotopy groups in degrees at least two.
\end{proof}

\subsection{\texorpdfstring{$E_2$}{E2} refinements of the specified multiplication and universal covers}

The specified products on $\cM$, $\cM_{\Q}$, and $\cM_{(p)}$ at every prime $p$ are not homotopy commutative.
Indeed, for a generator $x$ of the free part of $\pi_1(\cM)$, \eqref{rat:samelson-square} shows that $[x,x]$ is nonzero after rationalization, and hence also integrally and at every prime.
Since $A_\infty$ equivalences preserve Samelson products and $E_2$ multiplications are homotopy commutative, none of these products admits an $E_2$ refinement.
This requires no good-prime hypothesis; we now consider the lifted product on the universal cover.

\begin{theorem}[Criterion for an $E_2$ refinement of the universal cover]
\label{ainf:universal-cover-e2}\label{ainf:universal-cover-e2-equivalences}
Assume $p\nmid |H|$.
The specified multiplication on $(\widetilde\cM)_{(p)}$ admits an $E_2$ refinement if and only if $p\nmid |X|$.
If $L_{(p)}\simeq *$, there are $A_\infty$ equivalences
\begin{equation}\label{ainf:universal-cover-double-loop}
 (\widetilde\cM)_{(p)}
 \simeq_{A_\infty}\Omega(S^3_{(p)})
 \simeq_{A_\infty}\Omega^2\bigl((BS^3)_{(p)}\bigr).
\end{equation}
Here $S^3$ is the topological group of unit quaternions, and $BS^3$ is its classifying space.

If $L_{(p)}\not\simeq *$, the specified multiplication on $(\widetilde\cM)_{(p)}$ has infinite homotopy nilpotency in the sense of iterated commutator maps.
With $m,A$ as in Proposition~\ref{ainf:iterated-samelson}, iterated Samelson products realize $A$ as a direct summand of $\pi_{m+2n+1}((\widetilde\cM)_{(p)})$ for every $n\ge0$.
\end{theorem}

\begin{proof}
By Corollary~\ref{low:universal-cover}, $(\widetilde \cM)_{(p)}$ is $A_\infty$ equivalent to the simply connected cover of $\cM_{(p)}$: localization commutes with the $2$-connected cover of the simply connected delooping $Y$, as is seen on homotopy groups.

If $L_{(p)}\simeq *$, Theorem~\ref{loc:main} gives $\cM_{(p)}\simeq_{A_\infty}\Omega(S^2_{(p)})$.
Localizing the fiber sequence $S^3\to S^2\to K(\Z,2)$ associated with the Hopf bundle gives the $2$-connected cover of $S^2_{(p)}$.
Since $S^3\simeq_{A_\infty}\Omega BS^3$, this gives \eqref{ainf:universal-cover-double-loop} and hence an $E_2$ refinement.

Suppose $L_{(p)}\not\simeq *$ and use $\alpha,i,R,u,\operatorname{pr}_j$ from Proposition~\ref{ainf:iterated-samelson}.
Let $\gamma\in\pi_2(\cM_{(p)})$ be the image of the Hopf class of the sphere factor and, for $\beta\in A$, put $\beta'=[\alpha,i_*\beta]$.
Choosing the sign of the Hopf class gives
\begin{equation}\label{ainf:hopf-hurewicz-normalization}
 h_R(\gamma)=u^2.
\end{equation}
Indeed, let $\alpha_0\in\pi_1(\Omega S^2)$ be the standard generator and put $u_0=h_{\Z}(\alpha_0)$.
The Whitehead-square identity used in \eqref{rat:samelson-square}, together with $h_{\Z}[\alpha_0,\alpha_0]=2u_0^2$ and $H_2(\Omega S^2;\Z)\cong\Z$, gives this normalization integrally.
Naturality under localization and inclusion gives the displayed formula; this integral normalization also applies when $p=2$.

The classes $\gamma$ and $\beta'$ have degrees $2$ and $m+1$, respectively, and hence have unique lifts $\widetilde\gamma$ and $\widetilde{\beta'}$ to the universal cover, including when $m=1$.
Put $b_\beta=h_R(\beta)$.
Apply the projection \eqref{ainf:homology-projection} to
$h_R(\operatorname{ad}_\gamma^n\beta')=\operatorname{ad}_{u^2}^{n}([u,b_\beta])$.
As in Proposition~\ref{ainf:iterated-samelson}, its only term not ending in a positive power of $u$ is $u^{2n+1}b_\beta$, with coefficient one. Hence
\begin{equation}\label{ainf:cover-one-block-retraction}
 \operatorname{pr}_{2n+1}h_R\bigl(\operatorname{ad}_{\gamma}^{n}(\beta')\bigr)=\beta
 \qquad(n\ge0).
\end{equation}
For each $n\ge0$, the map $\beta\mapsto\operatorname{ad}_{\widetilde\gamma}^{n}(\widetilde{\beta'})$ is additive by bilinearity of Samelson products.
Composing the covering map, the Hurewicz homomorphism, and $\operatorname{pr}_{2n+1}$ gives a left inverse, proving the asserted direct-summand realization.
The resulting nonzero products of arbitrary length imply infinite homotopy nilpotency and exclude an $E_2$ refinement.
The equivalence with the arithmetic condition follows from Corollary~\ref{loc:arithmetic-nontriviality}.
\end{proof}

\subsection{Exponential growth of torsion homotopy groups}

The first nonzero group also gives an exponential lower bound for cyclic direct-summand multiplicities in each sufficiently large homotopy degree.
The tensor products below are over $\Z$; for finite abelian $p$-groups, tensoring over $\Z_{(p)}$ gives the same groups.
We use the terminology of $\Z/p^r$-hyperbolicity from \cite[Definition~2]{Boyde2022}.

\begin{theorem}[Exponential growth in each homotopy degree]
\label{ainf:general-exponential-torsion}
Assume the hypotheses of Proposition~\ref{ainf:iterated-samelson} and retain the notation $K,m,A$.
For $i,j\ge1$, let $c_{i,j}$ count the words in a Hall basis of the ordinary, ungraded free Lie algebra on two letters in which the letters occur $i$ and $j$ times, respectively.
Then the finite $p$-group $\pi_{i+(m+1)j}(\cM_{(p)})$ has a direct summand
\begin{equation}\label{ainf:general-torsion-summands}
 \bigl(A^{\otimes j}\bigr)^{c_{i,j}}.
\end{equation}
In particular, writing $\ell_n=c_{n,n}$, the group $\pi_{(m+2)n}(\cM_{(p)})$ contains $(A^{\otimes n})^{\ell_n}$ as a direct summand for every $n\ge1$; this summand is $C_p^{\ell_n h^n}$ if $A\cong C_p^h$ for some integer $h\ge1$.
Explicit formulas for $c_{i,j}$ and $\ell_n$ are given in Lemma~\ref{app:hall-count}.

Suppose that $A$ has a direct summand $C_{p^r}$, where $r\ge1$, and put $a=\nu_r(A)$, the number of its cyclic direct summands of order at least $p^r$.
Let $s_N(r)$ be the multiplicity of $C_{p^r}$ as a direct summand of
$\pi_N(\mathcal{R}_X)\otimes\Z_{(p)}$, for $N\ge3$.
Then $s_N(r)>0$ for all sufficiently large $N$, and
\begin{equation}\label{ainf:general-growth-rate}
 \liminf_{N\to\infty}\frac{\log s_N(r)}{N}
 \ge\frac{\log(4a)}{m+2}>0.
\end{equation}
The same bound holds for the cumulative multiplicity $T_N(r)$ of $C_{p^r}$ in $\pi_j(\mathcal{R}_X)\otimes\Z_{(p)}$ over $2\le j\le N$.
Thus the universal cover of $\mathcal{R}_X$ is $\Z/p^r$-hyperbolic, meaning that $\liminf_{N\to\infty}\log T_N(r)/N>0$.
\end{theorem}

\begin{proof}
Put $W=S^2\vee L_{(p)}$.
As in the proof of Proposition~\ref{hom:rack-shuffle}, the homotopy fiber $E$ of the pinch map $W\to L_{(p)}$ is
\begin{equation}\label{ainf:pinch-fibre}
 E\simeq S^2\wedge K_+\simeq S^2\vee\Sigma^2K
       =\Sigma(S^1\vee\Sigma K),
\end{equation}
where $K_+$ is $K$ with a disjoint basepoint, and the second equivalence is the suspension splitting.
The section $L_{(p)}\hookrightarrow W$ splits the homotopy exact sequence.
Since Theorem~\ref{loc:main} gives $\cM_{(p)}\simeq\Omega(W_{(p)})$, localization realizes $\pi_q(\Omega E)\otimes\Z_{(p)}$ as a direct summand of $\pi_q(\cM_{(p)})$ for $q\ge1$.

The Hilton--Milnor decomposition for the connected spaces $S^1$ and $\Sigma K$ \cite[Theorem~5.8]{Lavenir2025} gives one factor
\begin{equation}\label{ainf:hall-factor}
 \Omega\Sigma\bigl((S^1)^{\wedge i}\wedge(\Sigma K)^{\wedge j}\bigr)
\end{equation}
for each Hall word of multidegree $(i,j)$.
The space $K$ is $(m-1)$-connected with $H_m(K;\Z)=A$, also when $m=1$ since $\pi_1(K)$ is abelian.
For $d=i+(m+1)j$, connectivity and the K\"unneth theorem show that the smash product in \eqref{ainf:hall-factor} is $(d-1)$-connected with $H_d=A^{\otimes j}$; Tor terms occur in higher degrees.
Applying the Hurewicz theorem after suspension and then looping identifies the first nonzero homotopy group of \eqref{ainf:hall-factor} as $A^{\otimes j}$ in degree $d$.
The Hilton--Milnor projections split these factors simultaneously, with only finitely many contributing in each degree because their connectivity increases with word length.
Localization preserves the splittings and the finite $p$-groups $A^{\otimes j}$, proving \eqref{ainf:general-torsion-summands}.
The ambient groups are finite by Lemma~\ref{ainf:higher-finiteness}, since $d\ge3$.

To obtain the growth estimate, fix a cyclic direct-sum decomposition of $A$ containing one summand $C_{p^r}$.
In $A^{\otimes n}$, choose this summand in the first factor and any of the $a$ summands of order at least $p^r$ in each remaining factor.
Distributivity of tensor products and $C_{p^r}\otimes C_{p^s}\cong C_{p^r}$ for $s\ge r$ give a direct summand $(C_{p^r})^{a^{n-1}}$.
For large $N$, write $N=(m+2)n+\delta$ with $0\le\delta\le m+1$.
Taking $(i,j)=(n+\delta,n)$ in \eqref{ainf:general-torsion-summands} and using the covering $\cM\to\mathcal R_X$ gives
\[
 s_N(r)\ge c_{n+\delta,n}a^{n-1}.
\]
Lemma~\ref{app:hall-count} gives $\log c_{n+\delta,n}/n\to\log4$ for each of the finitely many possible values of $\delta$.
Since $N/n\to m+2$, taking logarithms proves \eqref{ainf:general-growth-rate} in all sufficiently large degrees.
Finally, $T_N(r)\ge s_N(r)$, and passage to the universal cover preserves higher homotopy groups.
\end{proof}

\begin{corollary}[The cyclic family and cyclic Sylow subgroups]
\label{ainf:cyclic-direct-summands}\label{ainf:exponential-torsion}
Under the hypotheses of Theorem~\ref{cyc:sylow-homology}, let $p^r$ be the order of a Sylow $p$-subgroup and $e$ its inertial index.
For every $j\ge2e-2$, the group
$\pi_j(\cM_{(p)})\cong\pi_j(\mathcal{R}_X)\otimes\Z_{(p)}$ has
$C_{p^r}$ as a direct summand, realized by the iterated Samelson products of Proposition~\ref{ainf:iterated-samelson}.

Moreover, $\pi_{2en}(\cM_{(p)})$ has
$(C_{p^r})^{\ell_n}$ as a direct summand for $n\ge1$, and
\begin{equation}\label{ainf:growth-rate}
 \liminf_{N\to\infty}\frac{\log s_N(r)}{N}
 \ge\frac{\log4}{2e}.
\end{equation}
The same bound holds for $T_N(r)$, so the universal cover of $\mathcal R_X$ is $\Z/p^r$-hyperbolic.
This applies to the cyclic Alexander family $X_{p^r,\omega}$ of
Section~\ref{sec:cyclic} with $e=d$.
\end{corollary}

\begin{proof}
The homotopy exact sequence of \eqref{cyc:sylow-fibre} gives $m=2e-2\ge2$ and $A=C_{p^r}$.
Apply Proposition~\ref{ainf:iterated-samelson} and Theorem~\ref{ainf:general-exponential-torsion} with $a=1$, using the covering $\cM\to\mathcal R_X$ in degrees at least two.
\end{proof}

For the tetrahedral quandle, Theorem~\ref{tet:loop-model} gives $m=1$
and $A=C_2$.
Thus $C_2$ is a direct summand of $\pi_j(\cM_{(2)})$ for every $j\ge1$
and of $\pi_j(\mathcal{R}_X)\otimes\Z_{(2)}$ for every $j\ge2$.
Moreover, Theorem~\ref{ainf:general-exponential-torsion} gives
\[
 \liminf_{N\to\infty}\frac{\log s_N(1)}N
 \ge\frac{\log4}{3}.
\]

The splittings used in the growth argument concern the underlying spaces and their homotopy groups; they do not assert product decompositions preserving the specified multiplication.

\appendix

\section{Counting Hall words and asymptotic estimates}
\label{app:hall}

\begin{lemma}\label{app:hall-count}
Take a Hall basis of the ordinary, ungraded free Lie algebra on two
letters, and let $c_{i,j}$ count the words in which the first and second
letters occur $i$ and $j$ times, respectively.
For $i,j\ge1$,
\begin{equation}\label{app:multigraded-witt}
 c_{i,j}=\frac1{i+j}
 \sum_{d\mid\gcd(i,j)}\mu(d)
 \binom{(i+j)/d}{j/d}.
\end{equation}
Here $\mu$ is the M\"obius function.
In particular, $\ell_n=c_{n,n}$ is given by
\begin{equation}\label{ainf:hall-multiplicity}
 \ell_n=\frac1{2n}\sum_{e\mid n}\mu(e)
                 \binom{2n/e}{n/e}\qquad(n\ge1).
\end{equation}
For every fixed integer $\delta\ge0$,
\begin{equation}\label{app:hall-bounded-difference}
 c_{n+\delta,n}\sim
 \frac{2^\delta4^n}{2\sqrt\pi\,n^{3/2}}
 \qquad(n\to\infty).
\end{equation}
Consequently, $\log c_{n+\delta,n}/n\to\log4$.
\end{lemma}

\begin{proof}
Including $c_{1,0}=c_{0,1}=1$, the multigraded
Poincar\'e--Birkhoff--Witt identity is
\[
 (1-x-y)^{-1}
 =\prod_{i+j>0}(1-x^iy^j)^{-c_{i,j}}.
\]
Taking formal logarithms and comparing coefficients gives
\[
 \frac1{i+j}\binom{i+j}{j}
 =\sum_{d\mid\gcd(i,j)}\frac{c_{i/d,j/d}}d.
\]
Multiplication by $i+j$ and M\"obius inversion give
\eqref{app:multigraded-witt}.
For $i=n+\delta$ and $j=n$, put $k=2n+\delta$ and separate the term with $d=1$:
\[
 k c_{n+\delta,n}=\binom{2n+\delta}{n}+R_n,
 \qquad |R_n|\le k2^{k/2}.
\]
Indeed, each remaining divisor $d$ satisfies $d\ge2$, so the corresponding
binomial coefficient is at most $2^{k/d}\le2^{k/2}$, and there are at most $k$ divisors.
For fixed $\delta$, Stirling's formula gives
\[
 \binom{2n+\delta}{n}\sim\frac{2^\delta4^n}{\sqrt{\pi n}}.
\]
The error is exponentially smaller than this term; division by
$k\sim2n$ proves \eqref{app:hall-bounded-difference} and its consequences.
\end{proof}

\section*{Acknowledgements}

The author thanks Katsumi Ishikawa for helpful correspondence concerning his spectral-sequence approach to the homology of connected dihedral quandles.
The author used ChatGPT to assist with editing and reviewing the manuscript and takes responsibility for its final content.

\end{document}